\documentclass[11pt]{article}

\makeatletter
\renewcommand*\@fnsymbol[1]{\the#1}
\makeatother

\usepackage{amsmath}
\usepackage{amssymb}
\usepackage{amsfonts}
\usepackage{amsthm}
\usepackage[english]{babel}
\usepackage{latexsym}
\usepackage[hmargin=1.6cm,vmargin=2.0cm]{geometry}
\usepackage{enumerate}
\usepackage{nicefrac}
\usepackage{mathrsfs}
\usepackage[affil-it]{authblk}
\usepackage{datetime}
\usepackage{comment}
\usepackage{courier}
\usepackage{color}
\usepackage{graphicx}
\usepackage{epstopdf}
\usepackage{bbm}

\theoremstyle{plain}
\newtheorem{theorem}{Theorem}[section]
\newtheorem{lemma}[theorem]{Lemma}
\newtheorem{proposition}[theorem]{Proposition}
\newtheorem{corollary}[theorem]{Corollary}
\theoremstyle{definition}

\newtheorem{remark}[theorem]{Remark}
\newtheorem{example}[theorem]{Example}
\newtheorem*{assumption1}{Assumption 1}
\newtheorem*{assumption2}{Assumption 2}
\newtheorem*{assumption3}{Assumption 3}

\newcommand{\Interior}{\mathop{\rm int}\nolimits}
\newcommand{\Int}{\mathop{\rm int}\nolimits}
\newcommand{\relint}{\mathop{\rm ri}\nolimits}

\newcommand{\rec}{\mathop{\rm rec}\nolimits}
\newcommand{\Cl}{\mathop{\rm cl}\nolimits}
\newcommand{\bd}{\mathop{\rm bd}\nolimits}
\newcommand{\co}{\mathop{\rm co}\nolimits}
\newcommand{\cone}{\mathop{\rm cone}\nolimits}
\newcommand{\aff}{\mathop{\rm aff}\nolimits}
\newcommand{\lin}{\mathop{\rm lin}\nolimits}
\newcommand{\ext}{\mathop{\rm ext}\nolimits}
\newcommand{\Span}{\mathop{\rm span}\nolimits}

\newcommand{\Barr}{\mathop{\rm bar}\nolimits}
\newcommand{\e}{\varepsilon}
\newcommand{\VaR}{\mathop {\rm VaR}\nolimits}

\newcommand{\ES}{\mathop {\rm ES}\nolimits}

\newcommand{\one}{\mathbbm{1}}

\newcommand{\risk}{\rho}
\newcommand{\Risk}{{\mathcal E}}

\newcommand{\cH}{\mathcal H}
\newcommand{\cI}{\mathcal I}
\newcommand{\cF}{\mathcal F}
\newcommand{\cG}{\mathcal G}
\newcommand{\cC}{\mathcal C}

\newcommand{\cQ}{\mathcal Q}
\newcommand{\cR}{\mathcal R}
\newcommand{\cV}{\mathcal V}
\newcommand{\cA}{\mathcal A}
\newcommand{\cK}{\mathcal K}

\newcommand{\cM}{\mathcal M}
\newcommand{\cN}{\mathcal N}
\newcommand{\cS}{\mathcal S}
\newcommand{\cX}{\mathcal X}
\newcommand{\cY}{\mathcal Y}
\newcommand{\cB}{\mathcal B}

\newcommand{\cU}{\mathcal U}

\newcommand{\probp}{\mathbb P}
\newcommand{\probq}{\mathbb Q}

\newcommand{\Riskeps}{\Risk_\varepsilon}

\newcommand{\R}{\mathbb R}
\newcommand{\N}{\mathbb N}
\newcommand{\E}{\mathbb E}

\newcommand{\refmultiple}{\cite[Proposition~1, 2, 3]{FarkasKochMunari2015}}
\newcommand{\refreductionlemma}{\cite[Lemma~2, 3]{FarkasKochMunari2015}}

\begin{document}

\title{
Existence, uniqueness and stability of\\
optimal portfolios of eligible assets
}

\author{
\sc{Michel Baes}
}
\affil{
Department of Mathematics, ETH Zurich, Switzerland\\
\texttt{mbaes@math.ethz.ch}
}
\author{
\sc{Pablo Koch-Medina},\,\sc{Cosimo Munari}
}
\affil{Center for Finance and Insurance and Swiss Finance Institute, University of Zurich, Switzerland\\
\texttt{pablo.koch@bf.uzh.ch},\,\,\texttt{cosimo.munari@bf.uzh.ch}
}

\date{\usdate\today}

\maketitle

\abstract{
In a capital adequacy framework, risk measures are used to determine the minimal amount of capital that a financial institution has to raise and invest in a portfolio of pre-specified eligible assets in order to pass a given capital adequacy test. From a capital efficiency perspective, it is important to identify the set of portfolios of eligible assets that allow to pass the test by raising the least amount of capital. We study the existence and uniqueness of such optimal portfolios as well as their sensitivity to changes in the underlying capital position. This naturally leads to investigating the continuity properties of the set-valued map associating to each capital position the corresponding set of optimal portfolios. We pay special attention to lower semicontinuity, which is the key continuity property from a financial perspective. This ``stability'' property is always satisfied if the test is based on a polyhedral risk measure but it generally fails once we depart from polyhedrality even when the reference risk measure is convex. However, lower semicontinuity can be often achieved if one if one is willing to focuses on portfolios that are close to being optimal. Besides capital adequacy, our results have a variety of natural applications to pricing, hedging, and capital allocation problems.
}

\bigskip


\medskip




\parindent 0em \noindent


\section{Introduction}

This paper is concerned with the study of a class of set-valued maps that play a natural and important role in several areas of mathematical finance. For the sake of definiteness, we introduce our mathematical setting focusing on a specific field of application, namely capital adequacy. The link to other financial problems is highlighted below.

\subsubsection*{Acceptance sets, eligible assets, risk measures}

Financial institutions are required by regulators to hold an adequate capital buffer to protect liability holders from the risk of default should severe unexpected losses occur. Whether the capital base of an institution is sufficient or not is established by a regulatory capital adequacy test that is usually based on Value at Risk (Basel 2-3, Solvency 2) or Expected Shortfall (Basel 4, Swiss Solvency Test). Associated to this test is a risk measure, or capital requirement rule, that determines the minimum amount of capital an institution would have to {\em raise} to pass the regulatory test. This minimum amount will, however, depend on how this capital is {\em invested} once raised. Hence, the risk measure makes sense as a rule to determine capital requirements only once a set of admissible investments has been specified. In line with the original framework described in the seminal paper by Artzner et al.~(1999), the bulk of the literature on risk measures assumes, either explicitly or implicitly, that raised capital is held either in cash or invested in a single pre-specified traded asset, which, using the language of Artzner et al.~(2009), we call the {\em eligible asset}. It is worth emphasizing that raising the amount of capital determined by the risk measure only ensures acceptability of the institution if this amount is actually invested in the eligible asset.

\medskip

As has been pointed out in Artzner et al.~(2009) and, more recently, in Farkas et al.~(2015), limiting the investment choice to a {\em single} eligible asset instead of allowing investments in {\em portfolios} of multiple eligible assets is inefficient in that it generally leads to higher capital requirements. Note that a lower capital requirement does not mean that the institution is less safe. Indeed, the institution continues to satisfy the same capital adequacy standard. It is just that to reach acceptability less capital is needed if more investment choices are allowed.

\medskip

From a mathematical perspective, a capital adequacy test is represented by an acceptance set for capital positions. In the modelling process, one starts by specifying a {\em space of capital positions} $\cX$ (typically a space of random variables). Each element of $\cX$ represents the capital (assets net of liabilities) of a company at a pre-specified future date (time $1$). The set of capital positions that are acceptable from a regulatory perspective is represented by a suitable set $\cA\subset\cX$, which is called the {\em acceptance set}. Hence, a company with capital position $X\in\cX$ is deemed adequately capitalized if, and only if, $X$ belongs to $\cA$.

\medskip

If the capital position of a financial institution is not acceptable, then management needs to undertake remedial actions to reach acceptability. The remedial action considered in this paper is raising new capital (at time $0$) and investing it in a portfolio of $N$ frictionless and liquidly traded assets
\[
S^1=(S^1_0,S^1_1),\dots,S^N=(S^N_0,S^N_1),
\]
which are referred to as the {\em eligible assets}. For any $i\in\{1,\dots,N\}$ the quantity $S^i_0\in\R$ represents the price of one unit of the $i$th asset at time $0$ and $S^i_1\in\cX$ represents the payoff of one unit of the same asset at time $1$. We denote by $\cM$ the vector subspace of $\cX$ spanned by the above payoffs, i.e.
\[
\cM = \Span(S^1_1,\dots,S^N_1).
\]
The space $\cM$ is called the {\em space of eligible payoffs}. Every element in $\cM$ can be viewed as the payoff of a suitable portfolio of eligible assets. If the Law of One Price holds, so that every portfolio generating the same payoff has the same initial price, one can define a {\em pricing functional} $\pi:\cM\to\R$ by setting
\[
\pi(Z)=\sum_{i=1}^N\lambda^iS^i_0, \ \ \ \mbox{for any $\lambda\in\R^N$ such that} \ Z=\sum_{i=1}^N\lambda^iS^i_1.
\]

\smallskip

The minimal amount of capital that needs to be raised and invested in a portfolio of eligible payoffs to pass the prescribed acceptability test is represented by the {\em risk measure} $\risk:\cX\to\overline{\R}$ defined by
\[
\risk(X) = \inf\{\pi(Z) \,; \ Z\in\cM, \ X+Z\in\cA\}.
\]
By definition, the quantity $\risk(X)$ admits a natural operational interpretation as a {\em capital requirement}. The risk measures introduced by Artzner et al.~(1999) constitute the prototype of the above capital requirement functionals and correspond to the simplest form of management action, i.e.~raising capital and investing it in a single eligible asset. The extension to a multi-asset framework was first dealt with, to the best of our knowledge, in F\"{o}llmer and Schied (2002) and later taken up in Frittelli and Scandolo (2006), Artzner et al.~(2009) and in the more comprehensive study by Farkas et al.~(2015).

\subsubsection*{Optimal eligible payoffs}

This paper is concerned with the study of the set-valued map $\Risk:\cX\rightrightarrows\cM$ defined by
\[
\Risk(X) = \{Z\in\cM \,; \ X+Z\in\cA, \ \pi(Z)=\risk(X)\}.
\]
The set $\Risk(X)$ consists of all the eligible payoffs that ensure acceptability of the position $X$ at the minimal cost. The map $\Risk$ will be therefore called the {\em optimal payoff map} and every payoff in $\Risk(X)$ will be referred to as an {\em optimal payoff} for $X$. In this paper we address the following theoretical questions that are critical also from an operational perspective:
\begin{itemize}
  \item {\bf Existence of optimal payoffs}. Do optimal payoffs of eligible assets exist at all? This is equivalent to assessing if $\Risk$ is not empty valued.
  \item {\bf Uniqueness of optimal payoffs}. In case optimal payoffs exist, are they unique or does management have several alternatives from which to choose? This is equivalent to assessing if $\Risk$ associates a unique payoff to each position.
  \item {\bf Stability of optimal payoffs}. In case several optimal payoffs exist, how robust is the choice of a specific portfolio?  More specifically, we are interested in (1) assessing if an optimal allocation remains close to being optimal after a slight perturbation of the underlying position and (2) ensuring that the accuracy of the optimal allocation increases with the degree of approximation of the underlying position. This is equivalent to investigating suitable continuity properties of $\Risk$: (1) corresponds to lower semicontinuity and (2) to upper semicontinuity.
  \item {\bf Stability of nearly-optimal payoffs}. If the choice of optimal payoffs is not robust, can we ensure robustness by relaxing the optimality condition and looking at portfolios of eligible assets that yield acceptability at a cost that is slightly higher than optimal? This question is of practical relevance because, in applications, one is bound to accept a certain tolerance for deviations from optimality.
\end{itemize}

\medskip

Any analysis of capital requirements would not be complete without having answered the above questions. From a theoretical perspective, failing to answer them would amount to studying an optimization problem focusing only on the optimal value of the objective function and not paying attention to the structure of the solution set. From a practical perspective, finding an answer to these questions is critical since, for a capital regime to be operationally effective, it is necessary to make sure that managers know which actions they can undertake to meet capital requirements and that these actions are robust with respect to misestimations.

\subsubsection*{Applications to pricing, hedging, and capital allocation}

The preceding questions are relevant in a variety of other areas of mathematical finance. In the context of pricing in incomplete markets, where an element $X\in\cX$ is interpreted as the payoff of a non-replicable financial contract, the quantity
\[
\risk(-X) = \inf\{\pi(Z) \,; \ Z\in\cM, \ Z-X\in\cA\}
\]
can be naturally viewed as a {\em price (from a seller's perspective)}. In this case, the elements of the acceptance set $\cA$ represent acceptable replication errors. If $\cA$ is taken to be the set of positive elements of $\cX$ (provided $\cX$ is partially ordered), then we obtain the standard {\em superreplication price}. If $\cA$ contains also nonpositive elements so that imperfect superreplication may be acceptable, we are in the framework of {\em good deal pricing} or {\em pricing under acceptable risk}; see Cochrane and Saa-Requejo~(2000), Carr et al.~(2001) and Jaschke and K\"{u}chler~(2001). This pricing approach has recently gained renewed attention in the framework of {\em conic finance} developed by Madan and Cherny~(2010). We also refer to Arai and Fukasawa~(2014).

\medskip

The risk measures studied in this paper arise naturally also in a variety of risk minimization problems related to hedging. To see this, let $\cX$ be a space of random variables containing the constants and assume $X\in\cX$ represents a given future exposure. Then, one can rewrite $\risk$ as
\[
\risk(X) = \inf\{\rho_\cA(X-Z+\pi(Z)) \,; \ Z\in\cM\},
\]
where $\rho_\cA:\cX\to\overline{\R}$ is the simple (cash-based) single-asset risk measure defined by
\[
\rho_\cA(X) = \inf\{m\in\R \,; \ X+m\in\cA\}.
\]
This shows that $\risk(X)$ can be interpreted as the minimal level of risk, as measured by $\rho_\cA$, at which we can secure the exposure $X$ by setting up portfolios of eligible assets. Note also that
\[
\risk(X) = \inf\{\rho_\cA(X-Z)-\pi(Z) \,; \ Z\in\cM\},
\]
showing that $\risk$ can be expressed as the {\em infimal convolution} of $\rho_\cA$ and $-\pi$ (one can always extend $\pi$ to the entire $\cX$ by setting $\pi(X)=-\infty$ whenever $X\notin\cM$). For a thorough discussion on hedging via risk minimization and the role of infimal convolutions in this setting we refer to Barrieu and El Karoui~(2009) and the references therein. Infimal convolutions of risk measures has been studied, for instance, in Jouini et al.~(2008) and Filipovi\'{c} and Svindland~(2008). We also refer to the comprehensive discussion in R\"{u}schendorf~(2013).

\medskip

Finally, we highlight that functionals of the form $\risk$ have been recently studied in the context of capital allocation and systemic risk. In this setting, one interprets the elements of $\cX$ as $d$-dimensional random vectors whose components represent the capital positions of $d$ financial entities (different companies, subsidiaries of an individual company, different desks). Similarly, the elements of $\cM$ are interpreted as $d$-dimensional random vectors consisting of eligible payoffs (one could in principle allow for different spaces of eligible assets along the different components) and $\pi$ is given by the aggregated sum of the individual pricing functionals. Under these specifications, the quantity
\[
\risk(X) = \inf\left\{\sum_{i=1}^d\pi_i(Z_i) \,; \ (Z_1,\dots,Z_d)\in\cM, \ (X_1+Z_1,\dots,X_d+Z_d)\in\cA\right\}
\]
represents the smallest amount of capital that has to be raised and allocated, in the form of portfolios of eligible assets, across the various entities of the system to ensure acceptability. The {\em systemic risk measures} studied in Biagini et al.~(2015) have the above form. The {\em efficient cash-invariant allocation rules} introduced in the framework of systemic risk by Feinstein et al.~(2017) are also related to the above functionals. A variety of acceptance sets for multivariate positions has been discussed in Molchanov and Cascos~(2016). The special case corresponding to stand-alone acceptability based on Expected Shortfall is treated in Hamel et al.~(2013). The case of multivariate shortfall risk is thoroughly studied in Armenti et al.~(2017).

\subsubsection*{Our contribution}

While properties of risk measures themselves have been subjected to detailed scrutiny, the above four questions have not been systematically addressed before in the risk measure literature. A variety of results on existence (exactness in the language of infimal convolutions) and uniqueness of optimal payoffs have been established in the context of hedging and capital allocation with risk measures for univariate positions; see Barrieu and El Karoui~(2009) and the references therein. In this paper we provide a comprehensive discussion on existence and uniqueness in general spaces of positions, thus encompassing also the multivariate case. However, the most important and, to the best of our knowledge, novel contribution of the paper consists in the investigation of the critical question of stability. This question constitutes a classical problem in optimization, more specifically in the subfield of parametric optimization; see e.g.~Bank et al.~(1983) and the references therein. There, one would formulate the problem by focusing on the {\em constraint set mapping} $\cF:\cX\rightrightarrows\cM$ defined by
\[
\cF(X) = \{Z\in\cM \,; \ X+Z\in\cA\}.
\]

\smallskip

The risk measure $\risk$ and the optimal payoff map $\Risk$ can be expressed in terms of the map $\cF$ as
\[
\risk(X)=\inf\{\pi(Z) \,; \ Z\in\cF(X)\} \ \ \ \mbox{and} \ \ \ \Risk(X)=\{Z\in\cF(X) \,; \ \pi(Z)=\risk(X)\}.
\]
The risk measure $\risk$ thus corresponds to the {\em extreme value function} and the optimal payoff map $\Risk$ to the {\em optimal set mapping}. The study of continuity properties of optimal set mappings is a recurrent theme in parametric optimization that goes under the names of {\em qualitative stability} or {\em perturbation analysis}. It is well-established that, among the various notions of stability or continuity, the property of lower semicontinuity is the most challenging to deal with. Unfortunately, it turns out that none of the (few) standard results on lower semicontinuity can be applied to our setting; see Remark~\ref{rem: conditions for lsc}. However, the key property from a financial perspective is precisely lower semicontinuity in that it ensures that a small perturbation in the underlying position does not result in a dramatic change in the structure of the optimal portfolio. To address the stability question we are thus forced to develop a variety of new results that exploit the special structure of our optimal set mapping.

\medskip

The paper is structured as follows. In Section~2 we introduce the underlying model space and list some relevant examples. In Section~3 we discuss the basic properties of optimal payoff maps. The question on existence and uniqueness of optimal payoffs is dealt with in Section~4. In particular, Proposition~\ref{prop: existence solutions asymptotic cone} and Corollary~\ref{cor: existence solutions star shaped} highlight the link between existence and the absence of (scalable) good deals. Section~5 is devoted to the question on stability. After focusing on outer and upper semicontinuity, we provide a comprehensive discussion on the key property of lower semicontinuity. As illustrated by a number of examples, lower semicontinuity may fail even in the presence of common acceptance sets. However, by Theorem~\ref{theo: lower semicontinuity polyhedral}, we always have lower semicontinuity if the chosen acceptance set is polyhedral. In the last part of the section we relax the optimality condition and focus on nearly-optimal payoffs. As a consequence of Theorem~\ref{theo: epsilon lower semicontinuity}, we are able to establish several sufficient conditions for nearly-optimal payoff maps to be lower semicontinuous.

\section{The underlying model space}

In this section we describe our model for the space of positions and the space of eligible payoffs and introduce some notation and terminology. In order to cover all special spaces of univariate and multivariate positions encountered in the literature we work in the context of an abstract space. This also helps highlight the fundamental mathematical structure of our problem.

\subsubsection*{The space of positions}

We consider a one-period model where financial positions at the terminal date are represented by elements of a Hausdorff topological vector space over $\R$, which we denote by $\cX$. We assume that $\cX$ is first countable and locally convex. This means that every element of $\cX$ admits a countable neighborhood base consisting of convex sets. In particular, $\cX$ could be any normed space. The topological dual of $\cX$ is denoted by $\cX'$.

\medskip

The {\em interior}, the {\em closure} and the {\em boundary} of a set $\cA\subset\cX$ are denoted by $\Int\cA$, $\Cl\cA$ and $\bd\cA$, respectively. We say that $\cA$ is {\em star-shaped (about zero)} whenever $X\in\cA$ implies that $\lambda X\in\cA$ for every $\lambda\in[0,1]$. The set $\cA$ is said to be {\em convex} if $\lambda X+(1-\lambda)Y\in\cA$ for any $X,Y\in\cA$ and $\lambda\in[0,1]$ and a {\em cone} if $\lambda X\in\cA$ for any $X\in\cA$ and $\lambda\in[0,\infty)$. Moreover, we say that $\cA$ is {\em strictly convex}  whenever $\lambda X+(1-\lambda)Y\in\Int\cA$ for all $\lambda\in(0,1)$ and any distinct $X,Y\in\cA$. Clearly, every convex set containing zero is star-shaped. Similarly, every (not necessarily convex) cone is star-shaped. The smallest convex set containing $\cA$ is called the {\em convex hull} of $\cA$ and is denoted by $\co\cA$. The smallest cone containing $\cA$ is called the {\em conic hull} of $\cA$ and is denoted by $\cone\cA$.  The {\em affine hull} of $\cA$, denoted by $\aff\cA$, is the set of all linear combinations of the form $\sum_{i=1}^m\alpha_iX_i$ for $X_1,\dots,X_m\in\cA$ and $\alpha_1,\dots,\alpha_m\in\R$ summing up to $1$. We say that $\cA$ is {\em polyhedral} if it can be represented as a finite intersection of halfspaces, i.e.
\[
\cA = \bigcap_{i=1}^{m}\{X\in\cX \,; \ \varphi_i(X)\geq\alpha_i\}
\]
for suitable functionals $\varphi_1,\dots,\varphi_m\in\cX'$ and scalars $\alpha_1,\dots,\alpha_m\in\R$. In this case, we say that $\cA$ is {\em represented} by $\varphi_1,\dots,\varphi_m$; see also~\eqref{eq: dual representation polyhedral} below. Clearly, any polyhedral set is closed and convex.

\medskip

The {\em asymptotic cone} of $\cA$ is the closed cone defined by setting
\[
\cA^\infty := \bigcap_{\e>0}\Cl\{\lambda X \,; \ \lambda\in[0,\e], \ X\in\cA\}.
\]
Equivalently, $\cA^\infty$ consists of all the limits of sequences $(\lambda_nX_n)$ where $(\lambda_n)\subset[0,\infty)$ with $\lambda_n\to0$ and $(X_n)\subset\cA$. If $\cA$ is closed and either convex or star-shaped, the asymptotic cone of $\cA$ coincides with the (otherwise smaller) {\em recession cone} of $\cA$ defined by
\[
\rec\cA := \{X\in\cX \,; \ Y+\lambda X\in\cA, \ \forall Y\in\cA, \ \forall \lambda\in(0,\infty)\}.
\]
The {\em lineality space} of $\cA$ is the vector space defined by $\lin\cA := \cA^\infty\cap(-\cA^\infty)$.

\smallskip

We assume that $\cX$ is partially ordered by a reflexive, antisymmetric, transitive relation $\geq$. The corresponding {\em positive cone} is given by
\[
\cX_+ := \{X\in\cX \,; \ X\geq0\}.
\]
In this case, the dual space $\cX'$ can be also partially ordered by setting $\varphi\geq\psi$ if and only if $\varphi(X)\geq\psi(X)$ for all $X\in\cX_+$. The associated positive cone is
\[
\cX'_+ := \{\varphi\in\cX' \,; \ \varphi(X)\geq0, \ \forall X\in\cX_+\}.
\]
An element $X\in\cX_+$ is said to be {\em strictly positive} whenever $\varphi(X)>0$ for all nonzero $\varphi\in\cX'_+$. The set of strictly-positive elements is denoted by $\cX_{++}$. Similarly, a functional $\varphi\in\cX'_+$ is called {\em strictly positive} whenever $\varphi(X)>0$ for all nonzero $X\in\cX_+$.

\smallskip

The {\em (lower) support function} of a set $\cA\subset\cX$ is the map $\sigma_\cA:\cX'\to\R\cup\{-\infty\}$ defined by
\[
\sigma_\cA(\varphi) := \inf_{X\in\cA}\varphi(X).
\]
The effective domain of $\sigma_\cA$, which is easily seen to be a convex cone, is called the {\em barrier cone} of $\cA$ and is denoted by $\Barr\cA$. We say that $X\in\cA$ is a {\em support point} for $\cA$ if there exists a functional $\varphi\in\cX'$ satisfying $\varphi(X)=\sigma_\cA(\varphi)$. Every support point for $\cA$ automatically belongs to $\bd\cA$. Although boundary points need not be support points, this is always true whenever $\cA$ is polyhedral or $\dim(\cX)<\infty$.

\medskip

We collect in the following result a variety of useful properties of support functions and barrier cones; see Aliprantis and Border~(2006) and, for point~{\em (iii)}, Farkas et al.~(2014). The lemma, in particular, tells us that the barrier cone of a vector space, respectively a cone, coincides with its annihilator, respectively its (one-sided) polar.

\begin{lemma}
\label{prop: properties support functions}
The following statements hold for any subsets $\cA,\cB\subset\cX$:
\begin{enumerate}[(i)]
\item $\sigma_{\cA+\cB}(\psi)=\sigma_\cA(\psi)+\sigma_\cB(\psi)$ for every $\psi\in\cX'$.
\item $\Barr(\cA+\cB)=\Barr\cA\cap\Barr\cB$.
\item If $\cA+\cX_+\subset\cA$, then $\Barr\cA\subset\cX'_+$.
\item If $\cA$ is a vector space, then $\Barr\cA=\{\varphi\in\cX' \,; \ \varphi(X)=0, \ \forall X\in\cA\}$.
\item If $\cA$ is a cone, then $\Barr\cA=\{\varphi\in\cX' \,; \ \varphi(X)\geq0, \ \forall X\in\cA\}$.
\item If $\cA$ is closed and convex, then
\[
\cA = \bigcap_{\varphi\in\Barr\cA}\{X\in\cX \,; \ \varphi(X)\geq\sigma_\cA(\varphi)\}.
\]
\item If $\cA$ is polyhedral and is represented by $\varphi_1,\dots,\varphi_m\in\cX'$, then $\varphi_1,\dots,\varphi_m\in\Barr\cA$ and
\begin{equation}
\label{eq: dual representation polyhedral}
\cA = \bigcap_{i=1}^m\{X\in\cX \,; \ \varphi_i(X)\geq\sigma_\cA(\varphi_i)\}.
\end{equation}
\end{enumerate}
\end{lemma}

\subsubsection*{The space of eligible payoffs}

The eligible payoffs are assumed to belong to a vector subspace $\cM\subset\cX$ with $1<\dim(\cM)<\infty$. We equip $\cM$ with the relative topology inherited from $\cX$. Since $\cM$ has finite dimension, the relative topology is always normable. In what follows we denote by $\|\cdot\|$ a fixed norm on $\cM$ and use the notation
\[
d(\cS_1,\cS_2) := \inf\{\|W-Z\| \,; \ W\in\cS_1, \ Z\in\cS_2\}
\]
for any subsets $\cS_1,\cS_2\subset\cM$. In addition, we write
\[
\cB_r(Z):=\{W\in\cM \,; \ \|W-Z\|\leq r\}
\]
for any $Z\in\cM$ and $r\in(0,\infty)$.

\medskip

Prices of eligible payoffs are represented by a linear functional $\pi:\cM\to\R$. Note that, being linear, $\pi$ is automatically continuous. The kernel of $\pi$ is denoted by
\[
\ker(\pi) := \{Z\in\cM \,; \ \pi(Z)=0\}.
\]
We assume throughout that $\cM$ contains a positive payoff with strictly-positive price (which, by linearity, can be always normalized to $1$).

\begin{assumption1}
We assume that there exists $U\in\cM\cap\cX_+$ such that $\pi(U)=1$.
\end{assumption1}

The above assumption is automatically satisfied if $\cM$ contains a nonzero positive payoff and the market is free of arbitrage opportunities so that $\pi$ is strictly positive.

\subsubsection*{The acceptance set}

The acceptance set is represented by a subset $\cA\subset\cX$ for which we stipulate the following assumptions.

\begin{assumption2}
We assume that $\cA$ is strictly contained in $\cX$ and satisfies:
\begin{enumerate}[(1)]
  \item $\cA$ is closed and contains zero.
  \item $\cA+\cX_+\subset\cA$.
\end{enumerate}
\end{assumption2}

\smallskip

The above properties are widely recognized as the minimal properties an acceptance set should satisfy. Property (1) is satisfied by all relevant acceptance sets in theory and practice. Property (2) is referred to as {\em monotonicity} and is equivalent to stipulating that any position that dominates an acceptable position should also be deemed acceptable.

\medskip

Note that we do not require a priori $\cA$ to be convex. The reason is that, in spite of its appealing interpretation in terms of diversification benefits, convexity is not satisfied by some relevant acceptance sets used in practice such as those based on Value at Risk. However, we need convexity of $\cA$ to establish some of our results. Within the class of convex acceptance sets, polyhedral sets and convex cones (sometimes called {\em coherent acceptance sets}) are particularly tractable.

\smallskip

\medskip

\begin{example}[{\bf Univariate acceptability}]
Assume $\cX$ is a standard space of random variables. For simplicity, we take $\cX=L^\infty(\Omega,\cF,\probp)$ and we consider the canonical almost-sure ordering on $\cX$.

\smallskip

(1) The {\em Value at Risk} (VaR) of a position $X\in\cX$ at the level $\alpha\in(0,1)$ is defined by
\[
\VaR_\alpha(X) := \inf\{m\in\R \,; \ \probp(X+m<0)\leq\alpha\}.
\]
Hence, up to a sign, $\VaR_\alpha(X)$ coincides with the upper $\alpha$-quantile of $X$. The corresponding acceptance set is the closed cone given by
\[
\cA = \{X\in\cX \,; \ \VaR_\alpha(X)\leq0\} = \{X\in\cX \,; \ \probp(X<0)\leq\alpha\}.
\]
In this case, acceptability boils down to checking whether the probability of default or loss of a certain position does not exceed the threshold $\alpha$. Note that $\cA$ is typically not convex. Acceptance sets based on VaR have historically been at the core of the Basel Accords, the reference regulatory regime for the banking sector, and of Solvency II, the regulatory regime for insurance companies within the European Union.

\smallskip

(2) The {\em Expected Shortfall} (ES) of $X\in\cX$ at the level $\alpha\in(0,1)$ is defined by
\[
\ES_\alpha(X) := \frac{1}{\alpha}\int_0^\alpha\VaR_\beta(X)d\beta.
\]
The corresponding acceptance set is the closed convex cone given by
\[
\cA = \{X\in\cX \,; \ \ES_\alpha(X)\leq0\}.
\]
As is well-known, $\cA$ is polyhedral whenever $\Omega$ is finite.
To be acceptable under ES, a financial institution needs to be solvent on average over the tail beyond the upper $\alpha$-quantile. Currently, ES is the basis for the Swiss Solvency Test, the regulatory framework for insurance companies in Switzerland, and is set to be the risk metric for market risk in the forthcoming Basel Accord.

\smallskip

(3) The acceptance set based on a nonempty set of {\em test scenarios} $E\in\cF$ is the closed convex cone
\[
\cA = \{X\in\cX \,; \ X\one_E\geq0\}.
\]
Here, we have denoted by $\one_E$ the indicator function of the event $E$. In this case, acceptability reduces to requiring solvency in each of the chosen test scenarios. Note that $\cA$ coincides with the positive cone of $\cX$ in the case that $E=\Omega$. Note also that $\cA$ is polyhedral whenever $\Omega$ is finite. The methodology used by many central exchanges and clearing counterparties to set up margin requirements, most notably the Standard Portfolio ANalysis of Risk adopted by the Chicago Mercantile Exchange, is based on test scenarios.

\smallskip

(4) The acceptance set based on a family $\cQ$ of probability measures on $(\Omega,\cF)$ that are absolutely continuous with respect to $\probp$ is the closed convex set defined by
\[
\cA = \bigcap_{\probq\in\cQ}\{X\in\cX \,; \ \E_\probq[X]\geq\alpha_\probq\}
\]
for suitable $\alpha_\probq\in(-\infty,0]$. The elements of $\cQ$ are often called {\em generalized scenarios}. Note that $\cA$ is polyhedral whenever $\cQ$ is finite and conic provided that $\alpha_\probq=0$ for all $\probq\in\cQ$. This type of acceptance sets is often encountered in the literature on good deal pricing or pricing with acceptable risk.

\smallskip

(5) The acceptance set based on an increasing and concave {\em utility function} $u:\R\to\R$ is given by
\[
\cA = \{X\in\cX \,; \ \E[u(X)]\geq\alpha\}
\]
for a suitable $\alpha\in(-\infty,u(0)]$. Note that $\cA$ is strictly convex whenever $u$ is strictly concave in the sense that $u(\lambda x+(1-\lambda)y)>\lambda u(x)+(1-\lambda)u(y)$ for all distinct $x,y\in\R$ and $\lambda\in(0,1)$. Acceptance sets based on expected utility are common in the literature on good deal pricing or pricing with acceptable risk.\hfill$\qed$
\end{example}

\smallskip

\begin{example}[{\bf Multivariate acceptability}]
Assume $\cX$ is a standard space of random vectors. For simplicity, we take $\cX=L_d^\infty(\Omega,\cF,\probp)$ for some $d\in\N$ and we consider the canonical componentwise almost-sure ordering on $\cX$.

\smallskip

(1) We speak of {\em stand-alone acceptability} whenever $\cA$ has the form
\[
\cA = \cC_1\times\cdots\times\cC_d
\]
where each component is an acceptance set in $L^\infty(\Omega,\cF,\probp)$. According to $\cA$, the system is adequately capitalized precisely when every entity is adequately capitalized on a stand-alone basis according to the corresponding univariate acceptance sets.

\smallskip

(2) The acceptance set based on an {\em aggregation function} $\Lambda:\R^d\to\R$ is defined by
\[
\cA = \{X\in\cX \,; \ \Lambda(X)\in\cC\}
\]
where $\cC$ is a given acceptance set in $L^\infty(\Omega,\cF,\probp)$. The function $\Lambda$ summarizes the system and its internal interdependencies in one figure that is tested against a univariate acceptance set.\hfill$\qed$
\end{example}

\medskip

In the framework of good deal pricing it is customary to assume that the market admits no {\em good deals}, i.e.~no payoff $Z\in\cA\cap\cM$ such that $\pi(Z)\leq0$. This is equivalent to requiring that every acceptable eligible payoff must have a strictly-positive price. In view of Assumptions 1-2, the absence of good deals is equivalent to
\[
\cA\cap\ker(\pi) = \{0\}.
\]
Sometimes, one is interested in ruling out those special good deals $Z\in\cM$ such that $\lambda Z\in\cA$ for all $\lambda>0$, which correspond to payoffs that can be purchased at zero or even negative cost and that are acceptable regardless of their size. Inspired by the terminology introduced in Pennanen~(2011), we refer to such payoffs as {\em scalable good deals}. In view of Assumptions 1-2, the absence of scalable good deals is equivalent to
\[
\cA^\infty\cap\ker(\pi) = \{0\}.
\]
The above conditions are important to ensure some results on existence and stability of optimal payoffs.

\subsubsection*{The risk measure}

The {\em risk measure} associated with the acceptance set $\cA$, the space of eligible payoffs $\cM$, and the pricing functional $\pi$ is the map $\risk:\cX\to\overline{\R}$ defined by setting
\[
\risk(X) := \inf\{\pi(Z) \,; \ Z\in\cM, \ X+Z\in\cA\}.
\]

\smallskip

For our study of existence and stability of optimal portfolios of eligible assets it is crucial to assume that $\rho$ take only finite values and be continuous; see Remark~\ref{rem: assumption continuity} for more details.

\begin{assumption3}
We assume that $\rho$ is finitely valued and continuous.
\end{assumption3}

\medskip

A variety of sufficient conditions for finiteness and continuity, which we record here for ease of reference, are provided in Farkas et al.~(2015). In particular, note that $\rho$ cannot be finitely valued (in fact, we would have $\risk\equiv-\infty$) unless
\[
\cA+\ker(\pi)\neq\cX.
\]
This condition was called {\em absence of acceptability arbitrage} in Farkas et al.~(2015) and says that it is not possible to make every position acceptable at zero cost.

\begin{proposition}[\refmultiple]
\label{prop: finiteness continuity rho}
Assume $\cA+\ker(\pi)\neq\cX$. Then, $\rho$ is finitely valued and continuous if any of the following conditions hold:
\begin{enumerate}[(i)]
  \item $\Int\cX_+\cap\cM\neq\emptyset$.
  \item $\cA$ is convex and $\cX_{++}\cap\cM\neq\emptyset$.
  \item $\cA$ is a convex cone and $\Int\cA\cap\cM\neq\emptyset$.
\end{enumerate}
\end{proposition}

\smallskip

\begin{remark}
\label{rem: no acceptability arbitrage}
The above conditions require the space of eligible payoffs to contain payoffs that are ``sufficiently risk-free''; see Farkas et al.~(2015) for more details. Note that condition {\em (i)} requires the positive cone $\cX_+$ to have nonempty interior. This is always the case if $\dim(\cX)<\infty$ but typically breaks down, with the exception of spaces of bounded random variables, in an infinite-dimensional setting.\hfill$\qed$
\end{remark}


\section{The optimal payoff map}

The {\em optimal payoff map} associated with the risk measure $\rho$ is the set-valued map $\Risk:\cX\rightrightarrows\cM$ defined by
\[
\Risk(X) := \{Z\in\cM \,; \ X+Z\in\cA, \ \pi(Z)=\risk(X)\}.
\]

\smallskip

Every eligible payoff in $\Risk(X)$ is called an {\em optimal payoff} for the position $X$. Note that, as mentioned above, there may exist positions $X$ such that $\Risk(X)$ is empty (even though $\risk(X)$ is finite).

\medskip

The next proposition lists some useful properties of the optimal payoff map that are used throughout the paper. To that effect, we first need to recall some basic properties of the risk measure $\risk$.

\begin{lemma}[\refreductionlemma]
\label{lem: properties rho}
For any $X,Y\in\cX$ the risk measure $\risk$ satisfies:
\begin{enumerate}[(i)]
  \item $X\geq Y$ implies $\risk(X)\leq\risk(Y)$.
  \item $\risk(X+Z)=\risk(X)-\pi(Z)$ for every $Z\in\cM$.
  \item $\risk(X)=\inf\{m\in\R \,; \ X+mU\in\cA+\ker(\pi)\}$.
  \item $\{X\in\cX \,; \ \rho(X)<0\}=\Interior(\cA+\ker(\pi))$.
  \item $\{X\in\cX \,; \ \rho(X)\leq0\}=\Cl(\cA+\ker(\pi))$.
  \item $\{X\in\cX \,; \ \rho(X)=0\}=\bd(\cA+\ker(\pi))$.
\end{enumerate}
\end{lemma}

\smallskip

\begin{proposition}
\label{prop: reformulation Risk}
For every $X\in\cX$ the following statements hold:
\begin{enumerate}[(i)]
  \item $\Risk(X)=\{Z\in\cM \,; \ X+Z\in\bd\cA\cap\bd(\cA+\ker(\pi))\}$.
  \item $\Risk(X+Z)=\Risk(X)-Z$ for every $Z\in\cM$.
  \item $\Risk(\cK)$ is closed for every compact set $\cK\subset\cX$.
  \item $\Risk(X)$ is convex whenever $\cA$ is convex.
  \item $\Risk(X)$ is polyhedral (in $\cM$) whenever $\cA$ is polyhedral.
  \item $\Risk(X)^\infty\subset\cA^\infty\cap\ker(\pi)$.
  \item $\Risk(X)^\infty=\cA^\infty\cap\ker(\pi)$ if $\cA$ is star-shaped and $\Risk(X)\neq\emptyset$.
\end{enumerate}
\end{proposition}
\begin{proof}
{\em (i)} Since $\cA$ is closed and $X+Z\in\bd(\cA+\ker(\pi))$ is equivalent to $\rho(X)=\pi(Z)$ for all $X\in\cX$ and $Z\in\cM$ by Lemma~\ref{lem: properties rho}, the assertion is established once we show that
\begin{equation}
\label{eq: reformulation Risk}
\Risk(X) \subset \{Z\in\cM \,; \ X+Z\in\bd\cA\}.
\end{equation}
To this effect, take any $Z\in\Risk(X)$ and note that, by definition, we have $X+Z\in\cA$ and $\risk(X)=\pi(Z)$. Should $X+Z\in\Interior\cA$ hold, we would find a suitable $\e>0$ such that $X+Z-\e U\in\cA$. However, this would imply that
\[
\risk(X) \leq \pi(Z)-\e\pi(U) < \risk(X).
\]
Hence, we must have $X+Z\in\bd\cA$ and this concludes the proof of~\eqref{eq: reformulation Risk}.

\smallskip

{\em (ii)} For any $Z\in\cM$ it follows from Proposition~\ref{prop: reformulation Risk} that
\[
\Risk(X+Z) = \bd\cA\cap\bd(\cA+\ker(\pi))-(X+Z) = (\bd\cA\cap\bd(\cA+\ker(\pi))-X)-Z = \Risk(X)-Z.
\]

\smallskip

{\em (iii)} Assume that $\cK\subset\cX$ is compact and consider a sequence $(Z_n)\subset\Risk(\cK)$ converging to some $Z\in\cM$ (recall that $\cM$ is closed). For any $n\in\N$ we find $X_n\in\cK$ such that $Z_n\in\Risk(X_n)$. Since $\cK$ is compact, a suitable subsequence $(X_{n_k})$ converges to some $X\in\cK$. Note that $X_{n_k}+Z_{n_k}\in\bd\cA\cap\bd(\cA+\ker(\pi))$ for all $k\in\N$ by virtue of Proposition~\ref{prop: reformulation Risk}. Note also that $Z_{n_k}\to Z$. As a result, we infer that $X+Z\in\bd\cA\cap\bd(\cA+\ker(\pi))$, which implies $Z\in\Risk(X)$ again by Proposition~\ref{prop: reformulation Risk}. This yields $Z\in\Risk(\cK)$ and concludes the proof.

\smallskip

{\em (iv)} Assume that $\cA$ is convex and take any $X\in\cX$. Then, for every $Z,W\in\Risk(X)$ and for every $\lambda\in[0,1]$ we clearly have
\[
X+\lambda Z+(1-\lambda)W = \lambda(X+Z)+(1-\lambda)(X+W) \in \cA
\]
as well as
\[
\pi(\lambda Z+(1-\lambda)W) = \lambda\risk(X)+(1-\lambda)\risk(X) = \risk(X).
\]
This shows that $\Risk(X)$ is convex.

\smallskip

{\em (v)} Assume $\cA$ is polyhedral so that we find suitable functionals $\varphi_1,\dots,\varphi_m\in\cX'_+$ satisfying
\[
\cA = \bigcap_{i=1}^{m}\{X\in\cX \,; \ \varphi_i(X)\geq\sigma_\cA(\varphi_i)\}.
\]
In this case, we easily see that
\begin{equation}
\label{eq: polyhedral structure Risk 1}
\{Z\in\cM \,; \ X+Z\in\cA\} = \bigcap_{i=1}^{m}\{Z\in\cM \,; \ \varphi_i(Z)\geq\sigma_\cA(\varphi_i)-\varphi_i(X)\}.
\end{equation}
Moreover, note that
\begin{equation}
\label{eq: polyhedral structure Risk 2}
\{Z\in\cM \,; \ \pi(Z)=\risk(X)\} = \{Z\in\cM \,; \ \pi(Z)\geq\risk(X)\}\cap\{Z\in\cM \,; \ -\pi(Z)\geq-\risk(X)\}.
\end{equation}
This shows that $\Risk(X)$ can be expressed as the intersection of two polyhedral sets in $\cM$ and is thus also polyhedral in $\cM$.

\smallskip

{\em (vi)} Take any $Z\in\Risk(X)^\infty$ so that $\lambda_n Z_n\to Z$ for a suitable sequence $(\lambda_n)\subset\R_+$ converging to zero and for $(Z_n)\subset\Risk(X)$. Since $\lambda_n(X+Z_n)\to Z$ and $X+Z_n\in\cA$ for every $n\in\N$, we see that $Z\in\cA^\infty$. Moreover, note that $Z$ belongs to $\cM$ (since $\cM$ is closed) and satisfies
\[
\pi(Z) = \lim_{n\to\infty}\lambda_n\pi(Z_n) = \lim_{n\to\infty}\lambda_n\risk(X) = 0
\]
by the continuity of $\pi$, so that $Z\in\ker(\pi)$. This proves that $\Risk(X)^\infty\subset\cA^\infty\cap\ker(\pi)$.

\smallskip

{\em (vii)} Recall that, if $\cA$ is star-shaped, we have $\cA^\infty=\rec\cA$. Moreover, recall that asymptotic cones always contain the corresponding recession cones. Hence, in light of point~{\em (v)}, the claim is established once we prove that
\begin{equation}
\label{eq: asymptotic cone Risk}
\rec\cA\cap\ker(\pi) \subset \rec\Risk(X).
\end{equation}
To this effect, take any $Z\in\rec\cA\cap\ker(\pi)$ and $W\in\Risk(X)$, which exists since $\Risk(X)$ is assumed to be nonempty. We claim that, for every $\lambda\in(0,\infty)$, we have $W+\lambda Z\in\Risk(X)$. To show this, note first that $X+W+\lambda Z\in\cA$. This follows from the fact that $Z\in\rec\cA$ and $X+W\in\cA$. Moreover, it is clear that
\[
\pi(W+\lambda Z) = \pi(W) = \risk(X).
\]
This shows that $W+\lambda Z\in\Risk(X)$ for every $\lambda\in(0,\infty)$ and establishes~\eqref{eq: asymptotic cone Risk}.
\end{proof}

\smallskip

\begin{remark}
We show that the assumptions in point~{\em (vii)} above are all necessary. Let $\cX=\R^2$ and $\cM=\cX$ and define the pricing functional $\pi$ by setting $\pi(X)=\frac{1}{2}(X_1+X_2)$ for all $X\in\cX$.
\begin{enumerate}[(i)]
  \item {\em $\cA$ is star-shaped but $\Risk(X)$ is empty}. Consider the acceptance set $\cA=\cA_1\cup\cA_2$ where
\[
\cA_1 = \{X\in\cX \,; \ X_1\in[0,\infty), \ X_2\in[-1,\infty)\}
\]
and
\[
\cA_2 = \{X\in\cX \,; \ X_1\in(-\infty,0), \ X_2\geq e^{X_1}-X_1-2\}.
\]

\smallskip

Note that $\cA$ is star-shaped. It is easy to verify that our assumptions (A1) to (A3) are all satisfied in this setting. Moreover, we have
\[
\cA+\ker(\pi) = \{X\in\cX \,; \ X_2>-X_1-2\}.
\]
Since $\bd\cA$ and $\bd(\cA+\ker(\pi))$ have empty intersection, it follows from Proposition~\ref{prop: reformulation Risk} that $\Risk(X)$ is empty for every $X\in\cX$. However, $\cA^\infty\cap\ker(\pi)$ contains infinitely many elements.
  \item {\em $\Risk(X)$ is nonempty but $\cA$ is not star-shaped}. Set $\alpha_n=-n+\frac{1}{n}$ for every $n\in\N$ and consider the acceptance set
\[
\cA = \cX_+\cup\bigcup_{n\in\N}\{X\in\cX \,; \ X_1\in[\alpha_{n+1},\alpha_n), \ X_2\in[n+1,\infty)\}.
\]
Note that $\cA$ is not star-shaped and that our assumptions (A1) to (A3) are all satisfied in this setting. Moreover, it is easy to verify that
\[
\cA+\ker(\pi) = \{X\in\cX \,; \ X_2\geq-X_1\}.
\]
Since $\Risk(0)=\{0\}$, we also have $\Risk(0)^\infty=\{0\}$. However, $\cA^\infty\cap\ker(\pi)$ is easily seen to contain infinitely many elements.\hfill$\qed$
\end{enumerate}
\end{remark}


\section{Existence and uniqueness of optimal payoffs}
\label{sec: existence uniqueness optimal payoffs}

This section is devoted to investigating under which conditions we can ensure existence and uniqueness of optimal eligible payoffs.


\subsubsection*{Existence of optimal payoffs}

We start by providing a general characterization of the global existence of optimal payoffs.

\begin{proposition}
\label{prop: existence solutions}
The following statements are equivalent:
\begin{enumerate}[(a)]
  \item $\Risk(X)\neq\emptyset$ for every $X\in\cX$.
  \item $\Risk(X)\neq\emptyset$ for every $X\in\bd(\cA+\ker(\pi))$.
  \item $\cA+\ker(\pi)$ is closed.
\end{enumerate}
\end{proposition}
\begin{proof}
It is clear that {\em (a)} implies {\em (b)}. Assume now that {\em (b)} holds but $\cA+\ker(\pi)$ is not closed, so that we find $X\in\bd(\cA+\ker(\pi))\setminus(\cA+\ker(\pi))$. This implies that $\Risk(X)\neq\emptyset$ and $\risk(X)=0$ by Lemma~\ref{lem: properties rho} but, at the same time, that there cannot exist $Z\in\ker(\pi)$ with $X+Z\in\cA$. Since this is not possible, we conclude that {\em (c)} must hold.

\smallskip

Finally, assume that {\em (c)} holds and take an arbitrary $X\in\cX$. Note that $\risk(X+\risk(X)U)=0$ and thus Lemma~\ref{lem: properties rho} implies that $X+\risk(X)U\in\cA+\ker(\pi)$. As a result, we find $Z\in\ker(\pi)$ such that $X+\risk(X)U+Z\in\cA$ and $\pi(\risk(X)U+Z)=\risk(X)$, proving that $\risk(X)U+Z\in\Risk(X)$. This shows that {\em (a)} holds and concludes the proof of the equivalence.
\end{proof}

\medskip

The preceding result shows that existence of optimal payoffs is equivalent to the ``augmented'' acceptance set $\cA+\ker(\pi)$ being closed. This set has a clear financial interpretation in that it consists of all the positions that can be made acceptable at zero cost, i.e.
\[
\cA+\ker(\pi) = \{X\in\cX \,; \ \mbox{$X+Z\in\cA$ for some $Z\in\ker(\pi)$}\}.
\]
In particular, when $\cA=\cX_+$, the above set is easily seen to coincide, up to a sign, with the set of positions that can be superreplicated at zero cost. Establishing the closedness of $\cA+\ker(\pi)$ is a recurrent theme in mathematical finance. In fact, this is a special case of a classical problem in functional analysis asking when the sum of two closed sets is still closed. As first remarked in Dieudonn\'{e}~(1966), the notion of asymptotic cone plays an important role to establish closedness.

\begin{proposition}
\label{prop: existence solutions asymptotic cone}
Assume the market admits no scalable good deals, i.e.~$\cA^\infty\cap\ker(\pi)=\{0\}$. Then, we have $\Risk(X)\neq\emptyset$ for every $X\in\cX$.
\end{proposition}
\begin{proof}
Since $\ker(\pi)^\infty=\ker(\pi)$, it follows from the absence of scalable good deals that $\cA^\infty\cap\ker(\pi)^\infty=\{0\}$. Moreover, being the subset of a finite-dimensional vector space, $\ker(\pi)$ is easily seen to be asymptotically compact in the sense of Barbu and Precupanu~(2012). Indeed, we can always find $\varepsilon>0$ and a neighborhood of zero $\cU\subset\cX$ such that $\Cl(\{\lambda X \,; \ \lambda\in[0,\varepsilon], \ X\in\ker(\pi)\}\cap\cU)$ is compact. As a result, we can apply Corollary~1.61 in Barbu and Precupanu~(2012) to conclude that $\cA+\ker(\pi)$ is closed. Proposition~\ref{prop: existence solutions} now yields that $\Risk(X)\neq\emptyset$ for every $X\in\cX$.
\end{proof}

\smallskip

\begin{remark}
Optimal payoffs are called {\em risk allocations} in the setting of multivariate shortfall risk measures studied by Armenti et al.~(2016). The main existence result is their Theorem~3.6, which provides a sufficient condition for the existence of risk allocations under a suitable assumption on the underlying multivariate loss function. This assumption is easily seen to be equivalent to the absence of scalable good deals.\hfill$\qed$
\end{remark}

\medskip

We apply the preceding proposition to star-shaped and polyhedral acceptance sets. First, we show that optimal payoffs always exist if the underlying acceptance set is star-shaped and the market does not admit good deals. Due to Assumption 2, this result applies to any convex or conic acceptance set.

\begin{corollary}
\label{cor: existence solutions star shaped}
Assume $\cA$ is star-shaped and the market admits no good deals, i.e.~$\cA\cap\ker(\pi)=\{0\}$. Then, we have $\Risk(X)\neq\emptyset$ for every $X\in\cX$.
\end{corollary}
\begin{proof}
It is not difficult to verify that any star-shaped set that is closed contains its asymptotic cone. As a result, the assertion follows at once from Proposition~\ref{prop: existence solutions asymptotic cone}.
\end{proof}

\medskip

Next, we show that every position admits optimal payoffs whenever the underlying acceptance set is polyhedral. In this case, we need not require the absence of (scalable) good deals.

\begin{corollary}
\label{cor: existence for polyhedral}
Assume $\cA$ is polyhedral. Then, we have $\Risk(X)\neq\emptyset$ for every $X\in\cX$.
\end{corollary}
\begin{proof}
The assertion follows from Proposition~\ref{prop: existence solutions} once we prove that $\cA+\ker(\pi)$ is closed. In fact, we show that $\cA+\ker(\pi)$ is even polyhedral. This is clear if $\dim(\cX)<\infty$ since, being a finite-dimensional space, $\ker(\pi)$ is polyhedral in a finite-dimensional setting and polyhedrality is preserved under addition. Hence, assume that $\dim(\cX)=\infty$ and note that we may suppose without loss of generality that $\dim(\ker(\pi))=1$ (otherwise proceed by a simple finite induction argument). We also assume that $\cA$ is represented by $\varphi_1,\dots,\varphi_m\in\cX'_+$. In view of Lemma~\ref{prop: properties support functions}, this means that
\[
\cA = \bigcap_{i=1}^m\{X\in\cX \,; \ \varphi_i(X)\geq\sigma_\cA(\varphi_i)\}.
\]
In this case, we readily have
\[
\cA^\infty = \bigcap_{i=1}^m\{X\in\cX \,; \ \varphi_i(X)\geq0\}.
\]
First, suppose that $\cA^\infty\cap\ker(\pi)\neq\{0\}$ and take a nonzero $Z\in\cA^\infty\cap\ker(\pi)$. Note that $Z\in\bd(\cA^\infty)$, for otherwise every $X\in\cX$ would admit $\lambda>0$ satisfying $\lambda X+Z\in\cA^\infty$, so that $\cX\subset\cA^\infty+\ker(\pi)$ (in contrast to the finiteness of $\risk$, see the discussion before Proposition~\ref{prop: finiteness continuity rho}). Since $Z$ is a boundary point of $\cA^\infty$, one can split $\{1,\dots,m\}$ into two subsets $I_0\neq\emptyset$ and $I_+$ such that
\[
\mbox{$\varphi_i(Z)=0$ for $i\in I_0$} \ \ \ \mbox{and} \ \ \ \mbox{$\varphi_i(Z)>0$ for $i\in I_+$}.
\]
In this case, we claim that
\[
\cA+\ker(\pi) = \bigcap_{i\in I_0}\{X\in\cX \,; \ \varphi_i(X)\geq\sigma_\cA(\varphi_i)\}.
\]
The inclusion ``$\subset$'' is clear. To show the converse inclusion, assume that $X\in\cX$ satisfies $\varphi_i(X)\geq\sigma_\cA(\varphi_i)$ for all $i\in I_0$ and note that we may always choose $\lambda>0$ large enough to satisfy
\[
\varphi_i(X+\lambda Z)=
\begin{cases}
\varphi_i(X)\geq\sigma_\cA(\varphi_i) & \mbox{if} \ i\in I_0,\\
\varphi_i(X)+\lambda\varphi_i(Z)\geq\sigma_\cA(\varphi_i) & \mbox{if} \ i\in I_+,
\end{cases}
\]
so that $X\in\cA+\ker(\pi)$. This proves the inclusion ``$\supset$'' and shows that $\cA+\ker(\pi)$ is polyhedral.

\smallskip

Now, suppose that $\cA^\infty\cap\ker(\pi)=\{0\}$. Take any nonzero $Z\in\ker(\pi)$ and note that $\varphi_j(Z)>0>\varphi_k(Z)$ for suitable $j,k\in\{1,\dots,m\}$ (otherwise either $Z$ or $-Z$ would need to belong to $\cA^\infty$). Hence, we may split $\{1,\dots,m\}$ into three subsets $I_0$, $I_+\neq\emptyset$ and $I_-\neq\emptyset$ such that
\[
\mbox{$\varphi_i(Z)=0$ for $i\in I_0$}, \ \ \ \mbox{$\varphi_i(Z)>0$ for $i\in I_+$}, \ \ \ \mbox{$\varphi_i(Z)<0$ for $i\in I_-$}.
\]
For any $j\in I_+$ and $k\in I_-$ define $\lambda_{jk}=-\tfrac{\varphi_k(Z)}{\varphi_j(Z)}>0$ so that the functional $\varphi_{jk}=\lambda_{jk}\varphi_j+\varphi_k$ belongs to $\ker(\pi)^\perp$. We claim that
\[
\cA+\ker(\pi) = \bigcap_{i\in I_0}\{X\in\cX \,; \ \varphi_i(X)\geq\sigma_\cA(\varphi_i)\}\cap
\bigcap_{j\in I_+,\,k\in I_-}\{X\in\cX \,; \ \varphi_{jk}(X)\geq\sigma_\cA(\varphi_{jk})\}.
\]
The inclusion ``$\subset$'' is clear. To prove the converse inclusion, take $X\in\cX$ and assume it belongs to the above finite intersection. We have to exhibit $\lambda\in\R$ such that $X+\lambda Z\in\cA$, or equivalently $\varphi_i(X+\lambda Z)\geq\sigma_\cA(\varphi_i)$ for all $i\in\{1,\dots,m\}$. To this effect, set
\[
\lambda = \max_{i\in I_+}\frac{\sigma_\cA(\varphi_i)-\varphi_i(X)}{\varphi_i(Z)}.
\]
Then, it is not difficult to verify that
\[
\varphi_i(X+\lambda Z)=
\begin{cases}
\varphi_i(X)\geq\sigma_\cA(\varphi_i) & \mbox{if} \ i\in I_0,\\
\varphi_i(X)+\lambda\varphi_i(Z)\geq\varphi_i(X)+\tfrac{\sigma_\cA(\varphi_i)-\varphi_i(X)}{\varphi_i(Z)}\varphi_i(Z)=\sigma_\cA(\varphi_i) & \mbox{if} \ i\in I_+,\\
\varphi_{ji}(X)-\lambda_{ji}\varphi_j(X+\lambda Z)\geq \lambda_{ji}\sigma_\cA(\varphi_j)+\sigma_\cA(\varphi_i)-\lambda_{ji}\sigma_\cA(\varphi_j)=\sigma_\cA(\varphi_i) & \mbox{if} \ i\in I_-,
\end{cases}
\]
where $j\in I_+$ satisfies $\lambda=\tfrac{\sigma_\cA(\varphi_j)-\varphi_j(X)}{\varphi_j(Z)}$ (note that $\varphi_j(X+\lambda Z)=\sigma_\cA(\varphi_j)$). This concludes the proof of the inclusion ``$\supset$'' and shows that $\cA+\ker(\pi)$ is polyhedral.
\end{proof}


\subsubsection*{Uniqueness of optimal payoffs}

After characterizing the existence of optimal payoffs, the next natural question is under which conditions we can ensure uniqueness. We start by showing a general characterization of uniqueness, which is a simple consequence of the definition of our optimal payoff map. Here, for $X\in\cX$ we denote by $|\Risk(X)|$ the cardinality of the set $\Risk(X)$.

\begin{proposition}
\label{prop: characterization uniqueness}
Assume $\Risk(X)\neq\emptyset$ for every $X\in\cX$. Then, the following statements are equivalent:
\begin{enumerate}[(a)]
  \item $|\Risk(X)|=1$ for every $X\in\cX$.
  \item $|\Risk(X)|=1$ for every $X\in\bd\cA\cap\bd\left(\cA+\ker(\pi)\right)$.
  \item $\bd\cA\cap\bd\left(\cA+\ker(\pi)\right)\cap\left(\bd\cA+(\ker(\pi)\setminus\{0\})\right) =\emptyset$.
  \item $\bd\cA\cap\left(\bd\cA+(\ker(\pi)\setminus\{0\})\right)\subset\Int(\cA+\ker(\pi))$.
\end{enumerate}
\end{proposition}
\begin{proof}
It is clear that {\em (a)} implies {\em (b)}. Now, assume that {\em (b)} holds but we find $X\in\bd\cA\cap\bd(\cA+\ker(\pi))$ and $Z\in\ker(\pi)\setminus\{0\}$ such that $X+Z\in\bd\cA$. Since $\risk(X)=0$ by Lemma~\ref{lem: properties rho}, we see that $\Risk(X)$ contains both the null payoff $0\in\cM$ and the nonzero payoff $Z\in\cM$ so that $|\Risk(X)|\ge 2$. Since this contradicts {\em (b)}, we conclude that {\em (b)} must imply {\em (c)}.

\smallskip

It is immediate to verify that {\em (c)} implies {\em (d)}. Finally, assume that condition {\em (d)} is satisfied but there exist $Z_1,Z_2\in\Risk(X)$ with $Z_1\neq Z_2$ for some $X\in\cX$. In particular, note that $Z_2-Z_1\in\ker(\pi)\setminus\{0\}$. Since Proposition~\ref{prop: reformulation Risk} implies that $X+Z_1\in\bd\cA\cap\bd(\cA+\ker(\pi))$, it follows that
\[
X+Z_2 = X+Z_1+Z_2-Z_1 \in \left((\bd\cA\cap\bd(\cA+\ker(\pi)))+\ker(\pi)\setminus\{0\}\right)\cap\bd\cA.
\]
However, this is incompatible with condition {\em (d)}, showing that {\em (d)} implies {\em (a)}.
\end{proof}

\medskip

We provide sufficient conditions for uniqueness in the case of a polyhedral acceptance set. To this effect, we freely use the dual representation of polyhedral sets recorded in Lemma~\ref{prop: properties support functions}.

\begin{proposition}
Assume $\cA$ is polyhedral and is represented by $\varphi_1,\dots,\varphi_m\in\cX'_+$. For any $X\in\cX$ consider the set
\[
I_\cA(X) = \{i\in\{1,\dots,m\} \,; \ \varphi_i(X)=\sigma_\cA(\varphi_i)\}.
\]
Then, the following statements are equivalent:
\begin{enumerate}[(a)]
  \item $|\Risk(X)|=1$ for every $X\in\cX$.
  \item $\ker(\pi)\cap\bigcap_{i\in I_\cA(X)}\ker(\psi_i)=\{0\}$ for every $X\in\bd\cA\cap\bd(\cA+\ker(\pi))$.
\end{enumerate}
\end{proposition}
\begin{proof}
First of all, recall from Corollary~\ref{cor: existence for polyhedral} that every position admits an optimal payoff so that $\Risk(X)$ is nonempty for all $X\in\cX$. In addition, note that for any position $X\in\cX$ we have $I_\cA(X)\neq\emptyset$ if and only if $X\in\bd\cA$ due to polyhedrality. To prove that {\em (a)} implies {\em (b)}, assume that condition {\em (b)} fails for $X\in\bd\cA\cap\bd(\cA+\ker(\pi))$ so that we find a nonzero $Z\in\ker(\pi)$ that belongs to $\ker(\varphi_i)$ for all $i\in I_\cA(X)$. In particular, note that
\[
\varphi_i(X+\lambda Z) = \varphi_i(X)+\lambda\varphi_i(Z) = \sigma_\cA(\varphi_i) \ \ \ \mbox{for} \ i\in I_\cA(X)
\]
for every $\lambda\in(0,\infty)$. Since $\varphi_i(X)>\sigma_\cA(\varphi_i)$ for $i\notin I_\cA(X)$, it is also clear that
\[
\varphi_i(X+\lambda Z) = \varphi_i(X)+\lambda\varphi_i(Z) \geq \sigma_\cA(\varphi_i) \ \ \ \mbox{for} \ i\notin I_\cA(X)
\]
for $\lambda\in(0,\infty)$ small enough. This implies that $X+\lambda Z\in\cA$ for $\lambda\in(0,\infty)$ small enough. Since $\risk(X)=0$, we conclude that $|\Risk(X)|>1$. This establishes that {\em (a)} implies {\em (b)}.

\smallskip

Conversely, assume that {\em (a)} does not hold so that $|\Risk(X)|>1$ for some $X\in\bd\cA\cap\bd(\cA+\ker(\pi))$ by Proposition~\ref{prop: characterization uniqueness}. Take two distinct $Z_1,Z_2\in\Risk(X)$ and set $Y=X+\tfrac{1}{2}(Z_1+Z_2)$. Since $\tfrac{1}{2}(Z_1+Z_2)\in\Risk(X)$ by convexity of $\cA$, it follows from Proposition~\ref{prop: reformulation Risk} that $Y$ belongs to $\bd\cA\cap\bd(\cA+\ker(\pi))$. Then, for every $i\in I_\cA(Y)$ we have
\[
\sigma_\cA(\varphi_i) \leq
\begin{cases}
\varphi_i(X+Z_1) = \varphi_i(Y)+\tfrac{1}{2}\varphi_i(Z_1-Z_2) = \sigma_\cA(\varphi_i)+\tfrac{1}{2}\varphi_i(Z_1-Z_2)\\
\varphi_i(X+Z_2) = \varphi_i(Y)+\tfrac{1}{2}\varphi_i(Z_2-Z_1) = \sigma_\cA(\varphi_i)+\tfrac{1}{2}\varphi_i(Z_2-Z_1)
\end{cases}
\]
where we used that both $X+Z_1$ and $X+Z_2$ belong to $\cA$ in the first inequality. This implies that $Z_1-Z_2\in\ker(\varphi_i)$ for all $i\in I_\cA(Y)$. Since $\pi(Z_1-Z_2)=\risk(X)-\risk(X)=0$, we conclude that condition {\em (b)} is violated. It follows that {\em (b)} implies {\em (a)}.
\end{proof}

\medskip

The above condition for uniqueness takes the following simple form for acceptance sets that are polyhedral cones.

\begin{corollary}
Assume $\cA$ is a polyhedral cone represented by $\varphi_1,\dots,\varphi_m\in\cX'_+$ and
\[
\ker(\pi)\cap\ker(\varphi_i)=\{0\} \ \ \ \mbox{for all} \ i\in\{1,\dots,m\}.
\]
Then, $|\Risk(X)|=1$ for every $X\in\cX$.
\end{corollary}

\medskip

We conclude this section by establishing another useful result on uniqueness, which states that uniqueness is always ensured whenever the acceptance set is ``strictly convex'' along the directions of $\ker(\pi)$.

\begin{proposition}
\label{prop: sufficient condition uniqueness}
Assume $\Risk(X)\neq\emptyset$ for all $X\in\cX$ and for any distinct $X,Y\in\bd\cA$ with $X-Y\in\ker(\pi)$ there exists $\lambda\in(0,1)$ satisfying $\lambda X+(1-\lambda)Y\in\Int\cA$. Then, $|\Risk(X)|=1$ for all $X\in\cX$.
\end{proposition}
\begin{proof}
Take any position $X\in\bd\cA\cap(\bd\cA+\ker(\pi)\setminus\{0\})$. The claim follows directly from Proposition~\ref{prop: characterization uniqueness} once we show that $X$ belongs to $\Int(\cA+\ker(\pi))$. To this effect, note that $X=Y+Z$ for a suitable position $Y\in\bd\cA$ and a nonzero payoff $Z\in\ker(\pi)$. Hence, by assumption, we find a scalar $\lambda\in(0,1)$ satisfying $\lambda X+(1-\lambda)Y\in\Int\cA$. Now, set for convenience
\[
M = X-(1-\lambda)Z = \lambda X+(1-\lambda)Y \in \Int\cA,
\]
so that $M+\cU\subset\cA$ for some neighborhood of zero $\cU\subset\cX$. Since every $W\in X+\cU$ is easily seen to satisfy $W-(1-\lambda)Z-M = W-X \in \cU$, it follows that
\[
X+\cU \subset (1-\lambda)Z+M+\cU \subset \cA+\ker(\pi).
\]
This shows that $X$ is an interior point of $\cA+\ker(\pi)$ and concludes the proof.
\end{proof}

\smallskip

\begin{remark}
The above sufficient condition for uniqueness is, in general, not necessary. To see this, let $\cX=\R^3$ and consider the polyhedral acceptance set given by
\[
\cA = \co(\{X^1,X^2,X^3\})+\cX_+
\]
where $X^1=(0,-1,1)$, $X^2=(-1,0,1)$, and $X^3=(-\tfrac{1}{2},0,0)$. Assume that $\cM=\cX$ and define $\pi$ by setting $\pi(X)=\tfrac{1}{3}(X_1+X_2+X_3)$ for all $X\in\cX$. It is immediate to verify that our assumptions (A1) to (A3) are all satisfied in this setting. Moreover, since
\[
\cA+\ker(\pi) = \{X\in\cX \,; \ X_1+X_2+X_3\geq-\tfrac{1}{2}\},
\]
it follows that $\Risk(X)$ is nonempty by Proposition~\ref{prop: existence solutions} and is easily seen to consist of a single payoff due to Proposition~\ref{prop: characterization uniqueness} for every $X\in\cX$. However, the positions $X^1$ and $X^2$ both belong to $\bd\cA$ and satisfy $X^1-X^2\in\ker(\pi)\setminus\{0\}$ and the entire segment connecting $X^1$ and $X^2$ lie in $\bd\cA$.\hfill$\qed$
\end{remark}

\medskip

As a direct consequence of Proposition~\ref{prop: sufficient condition uniqueness}, we infer that a position admits at most one optimal payoff if the underlying acceptance set is strictly convex.

\begin{corollary}
\label{cor: uniqueness strictly convex}
Assume $\cA$ is strictly convex and $\Risk(X)\neq\emptyset$ for all $X\in\cX$. Then, $|\Risk(X)|=1$ for all $X\in\cX$.
\end{corollary}

\smallskip

\begin{remark}
In the multivariate setting of Armenti et al.~(2016), the uniqueness of optimal payoffs is crucial to define a meaningful allocation of systemic risk across the entities of the financial system under investigation. The uniqueness result recorded in their Theorem~3.6 can be seen to be equivalent to the strict convexity of the underlying acceptance set.\hfill$\qed$
\end{remark}


\section{Stability of optimal payoffs}
\label{sec: continuity optimal payoff map}

The remainder of the paper focuses on the question of stability or robustness of optimal portfolios. As mentioned in the introduction, from an operational perspective it is critical to ensure that a slight {\em perturbation} of $X$, arising for instance from estimation or specification errors, does not alter the set of optimal payoffs in a significant way. In other words, for all positions $X,Y\in\cX$ we would like to ensure that
\[
\mbox{$Y$ is close to $X$} \ \implies \ \mbox{$\Risk(Y)$ is ``close'' to $\Risk(X)$}.
\]
In order to specify a notion of proximity between sets of optimal payoffs we are naturally led to investigate the (semi)continuity properties of the optimal payoff map. A variety of (semi)continuity notions for set-valued maps have been investigated in the literature and the same terminology is often used to capture different forms of continuity. The notions of upper and lower semicontinuity go back to Berge~(1963), but our notion of upper semicontinuity is slightly different and aligned with that of upper hemicontinuity in Aliprantis and Border~(2006). The notion of outer semicontinuity is taken from Rockafellar and Wets~(2009).

\bigskip

{\bf Outer semicontinuity}. For any given position $X\in\cX$ we say that $\Risk$ is {\em outer semicontinuous at $X$} if for every $Z\notin\Risk(X)$ we find open neighborhoods $\cU_X\subset\cX$ of $X$ and $\cU_Z\subset\cM$ of $Z$ such that
\[
Y\in\cU_X \ \implies \ \Risk(Y)\cap\cU_Z=\emptyset.
\]
We say that $\Risk$ is {\em outer semicontinuous (on $\cX$)} if $\Risk$ is outer semicontinuous at every $X\in\cX$. This stability property requires that an eligible payoff that is non-optimal for a given position cannot be optimal for any position that is sufficiently close to the given one.

\bigskip

{\bf Upper semicontinuity}. For any given position $X\in\cX$ we say that $\Risk$ is {\em upper semicontinuous at $X$} if for every open set $\cU\subset\cM$ with $\Risk(X)\subset\cU$ we find an open neighborhood $\cU_X\subset\cX$ of $X$ such that
\[
Y\in\cU_X \ \implies \ \Risk(Y)\subset\cU.
\]
We say that $\Risk$ is {\em upper semicontinuous (on $\cX$)} if $\Risk$ is upper semicontinuous at every $X\in\cX$. Intuitively speaking, upper semicontinuity ensures that the set of optimal payoffs does not suddenly ``explode'' as a result of a slight perturbation of the underlying position.

\bigskip

{\bf Lower semicontinuity}. For a position $X\in\cX$ we say that $\Risk$ is {\em lower (or inner) semicontinuous at $X$} if for every open set $\cU\subset\cM$ with $\Risk(X)\cap\,\cU\neq\emptyset$ we find an open neighborhood $\cU_X\subset\cX$ of $X$ such that
\[
Y\in\cU_X \ \implies \ \Risk(Y)\cap\,\cU\neq\emptyset.
\]
We say that $\Risk$ is {\em lower (or inner) semicontinuous (on $\cX$)} if $\Risk$ is lower semicontinuous at every $X\in\cX$. Intuitively speaking, lower semicontinuity ensures that the set of optimal payoffs does not suddenly ``shrink'' as a result of a slight perturbation of the underlying position. In this case, for every position $X$ the following intuitive robustness property is satisfied:
\[
\mbox{$Y$ is close to $X$ and $Z\in\Risk(X)$} \ \implies \ \mbox{there exists a payoff in $\Risk(Y)$ that is close to $Z$}.
\]

In other words, lower semicontinuity guarantees that an optimal payoff remains close to being optimal after a small perturbation of the underlying position. This continuity property therefore constitutes the key stability notion in our financial context.

\smallskip

\begin{remark}
\label{rem: semicontinuity}
(i) Since $\Risk$ is closed valued, it is immediate to see that upper semicontinuity always implies outer semicontinuity for our optimal payoff map $\Risk$.

\smallskip

(ii) Consider a position $X\in\cX$ with $|\Risk(X)|=1$. Then, $\Risk$ is lower semicontinuous at $X$ whenever it is upper semicontinuous at $X$. The converse is, however, not true. But if $|\Risk(X)|=1$ for all $X\in\cX$, then upper and lower semicontinuity are equivalent and boil down to the continuity of the map assigning to $X$ the unique element in $\Risk(X)$.

\smallskip

(iii) For set-valued maps that are convex valued, the property of lower semicontinuity is especially powerful because it ensures the existence of continuous selections; see Theorem~17.66 in Aliprantis and Border (2009). Recall that a function $\zeta:\cX\to\cM$ is a {\em selection} of $\Risk$ if it satisfies $\zeta(X)\in\Risk(X)$ for every $X\in\cX$ such that $\Risk(X)$ is nonempty. A continuous selection for the optimal payoff map can be therefore interpreted as a procedure to select optimal payoffs in a robust way.\hfill$\qed$
\end{remark}


\subsection*{Outer semicontinuity}

We start by showing that the optimal payoff map is always outer semicontinuous on the whole of $\cX$.

\begin{theorem}
The optimal payoff map $\Risk$ is outer semicontinuous.
\end{theorem}
\begin{proof}
Take an arbitrary $X\in\cX$ and fix $Z\in\cX\setminus\Risk(X)$. Assume first that $Z\notin\cM$. In this case, since $\cM$ is closed, we find a neighborhood $\cU_Z\subset\cM$ of $Z$ satisfying $\cU_Z\cap\cM=\emptyset$, which implies $\Risk(Y)\cap\cU_Z=\emptyset$ for every $Y\in\cX$. Assume now that $Z\in\cM$ but $X+Z\notin\cA$. Since $\cA$ is closed, there exist neighborhoods $\cU_X\subset\cX$ of $X$ and $\cU_Z\subset\cM$ of $Z$ such that $(\cU_X+\cU_Z)\cap\cA=\emptyset$. In particular, we must have $\Risk(Y)\cap\cU_Z=\emptyset$ for every $Y\in\cU_X$. Finally, assume that $Z\in\cM$ and $X+Z\in\cA$ but $\pi(Z)\neq\risk(X)$. In this case, set
\[
\e = \tfrac{1}{4}|\pi(Z)-\risk(X)|
\]
and consider the neighborhoods of $X$ and $Z$ defined, respectively, by (recall that $\risk$ and $\pi$ are continuous)
\[
\cU_X=\{Y\in\cX \,; \ |\risk(Y)-\risk(X)|<\e\} \ \ \ \mbox{and} \ \ \ \cU_Z=\{W\in\cM \,; \ |\pi(W)-\pi(Z)|<\e\}.
\]
Then, taking any element $W\in\Risk(Y)\cap\cU_Z$ with $Y\in\cU_X$ we obtain
\[
\e \leq \tfrac{1}{4}|\pi(Z)-\pi(W)|+\tfrac{1}{4}|\pi(W)-\risk(X)| = \tfrac{1}{4}|\pi(Z)-\pi(W)|+\tfrac{1}{4}|\risk(Y)-\risk(X)| < \tfrac{1}{4}\e+\tfrac{1}{4}\e = \tfrac{1}{2}\e,
\]
which is clearly impossible. This shows that $\Risk(Y)\cap\cU_Z=\emptyset$ must hold for every $Y\in\cU_X$. In conclusion, it follows that $\Risk$ is always outer semicontinuous at $X$.
\end{proof}


\subsection*{Upper semicontinuity}

This section provides necessary and sufficient conditions for the optimal payoff map to be upper semicontinuous. We start by proving a sequential characterization of upper semicontinuity based on the general characterization recorded in Theorem 17.20 in Aliprantis and Border~(2006). By exploiting the special structure of the optimal payoff map, we can sharpen that result and show that upper semicontinuity can be charaterized in terms of sequences without any compactness assumption. In fact, upper semicontinuity always implies compact-valuedness.

\begin{proposition}
\label{prop: characterization usc}
The following statements are equivalent:
\begin{enumerate}[(a)]
  \item $\Risk$ is upper semicontinuous.
  \item $\Risk(\cK)$ is bounded (in $\cM$) for every compact set $\cK\subset\cX$.
  \item For every $X\in\cX$ we have
\[
X_n\to X, \ Z_n\in\Risk(X_n) \ \implies \ \exists Z\in\Risk(X), \ \exists (Z_{n_k}) \,:\, Z_{n_k}\to Z.
\]
\end{enumerate}
\end{proposition}
\begin{proof}
First of all, assume that {\em (a)} holds. Recall from Lemma~17.8 in Aliprantis and Border~(2006) that any upper semicontinuous set-valued map that is compact valued automatically sends compact sets into bounded sets. Since $\Risk$ is closed valued by Proposition~\ref{prop: reformulation Risk}, assertion {\em (b)} follows once we prove that $\Risk(X)$ is bounded (in $\cM$) for any given $X\in\cX$. To this effect, assume that $\Risk(X)\neq\emptyset$ and consider the open neighborhood of $\Risk(X)$ in $\cM$ defined by
\[
\cU = \bigcup_{Z\in\Risk(X)\cap\cB_1(0)}\hspace{-0.5cm}\Int\cB_1(Z)\cup
\bigcup_{r>1}\bigcup_{Z\in\Risk(X)\cap\bd\cB_r(0)}\hspace{-0.5cm}\Int\cB_{\frac{1}{r}}(Z).
\]
Moreover, consider the strictly-increasing function $f:(1,\infty)\to(2,\infty)$ given by $f(r)=r+\frac{1}{r}$. Note that for every $W\in\cU$ with $\|W\|>2$ we find suitable $r>1$ and $Z\in\Risk(X)\cap\bd\cB_r(0)$ such that $W\in\Int\cB_{\frac{1}{r}}(Z)$, which implies $\|W\| \leq \|Z\|+\|W-Z\| \leq f(r)$ and yields
\[
\|W-Z\| \leq \frac{1}{r} \leq \frac{1}{f^{-1}(\|W\|)}.
\]
As a result, it follows that
\begin{equation}
\label{eq: outer semi implies compact}
W\in\cU\setminus\cB_2(0) \ \implies \ d(\{W\},\Risk(X))\leq\frac{1}{f^{-1}(\|W\|)}.
\end{equation}
By upper semicontinuity there exists an open neighborhood $\cU_X\subset\cX$ of $X$ such that $\Risk(Y)\subset\cU$ for all $Y\in\cU_X$. In particular, we find $\e>0$ small enough so that $X+\e U\in\cU_X$ and thus
\begin{equation}
\label{eq: outer semi implies compact 2}
\Risk(X)-\e U = \Risk(X+\e U) \subset \cU.
\end{equation}
We claim that
\begin{equation}
\label{eq: outer semi implies compact 3}
d(\Risk(X)-\e U,\Risk(X)) > 0.
\end{equation}

\smallskip

To see this, take any $Z,W\in\Risk(X)$ and note that, since $\pi(Z)=\risk(X)=\pi(W)$, we have $Z-W\in\ker(\pi)$. This yields
\[
\|Z-\e U-W\| \geq \e d(\{U\},\ker(\pi)) > 0,
\]

\smallskip

establishing~\eqref{eq: outer semi implies compact 3} (the last inequality holds by continuity of $\pi$ because $\pi(U)>0$). By combining~\eqref{eq: outer semi implies compact}, \eqref{eq: outer semi implies compact 2} and~\eqref{eq: outer semi implies compact 3}, we conclude that $\Risk(X)$ must be bounded. Indeed, if this were not the case, we could find by~\eqref{eq: outer semi implies compact 2} an unbounded sequence $(Z_n)\subset\Risk(X)-\e U\subset\cU$ satisfying
\[
d(\{Z_n\},\Risk(X)) \leq \frac{1}{f^{-1}(\|Z_n\|)} \to 0
\]
by~\eqref{eq: outer semi implies compact}. However, this would contradict~\eqref{eq: outer semi implies compact 3}. In conclusion, it follows that $\Risk(X)$ is bounded and, thus, {\em (b)} holds.

\smallskip

Next, assume that {\em (b)} holds and consider a sequence $(X_n)\subset\cX$ and $X\in\cX$ such that $X_n\to X$. Moreover, let $Z_n\in\Risk(X_n)$ be fixed for each $n\in\N$. Since, being convergent, $(X_n)$ is contained in a compact set, it follows from {\em (a)} that $(Z_n)$ is bounded and therefore admits a convergent subsequence $(Z_{n_k})$. Let $Z\in\cM$ be the corresponding limit (note that $\cM$ is closed). Since $X_{n_k}+Z_{n_k}\in\bd\cA\cap\bd(\cA+\ker(\pi))$ for all $k\in\N$ by Proposition~\ref{prop: reformulation Risk}, we easily infer that $X+Z\in\bd\cA\cap\bd(\cA+\ker(\pi))$ as well. Hence, we can apply Proposition~\ref{prop: reformulation Risk} to conclude that $Z\in\Risk(X)$. This establishes {\em (c)}. We complete the proof of the equivalence by recalling that {\em (c)} always implies {\em (a)} due to Theorem~17.20 in Aliprantis and Border~(2006).
\end{proof}

\medskip

The next result shows that $\Risk$ is always upper semicontinuous whenever the market for the eligible assets admits no scalable good deals.

\begin{corollary}
\label{cor: sufficient usc asymptotic}
Assume the market admits no scalable good deals, i.e.~$\cA^\infty\cap\ker(\pi)=\{0\}$. Then, $\Risk$ is upper semicontinuous.
\end{corollary}
\begin{proof}
The claim follows from Proposition~\ref{prop: characterization usc} once we show that, for any fixed compact set $\cK\subset\cX$, the set $\Risk(\cK)$ is bounded (in $\cM$). By way of contrast, assume that $\Risk(\cK)$ is not bounded so that we find a sequence $(X_n)\subset\cK$ and elements $Z_n\in\Risk(X_n)$ for $n\in\N$ such that $\|Z_n\|\geq n$ for every $n\in\N$. Without loss of generality we assume that $X_n\to X$ for some $X\in\cK$. Note that, by compactness, we can extract a suitable subsequence $(Z_{n_k})$ such that $\tfrac{1}{\|Z_{n_k}\|}(Z_{n_k})\to Z$ for some nonzero $Z\in\cM$ (recall that $\cM$ is closed). In particular, we also have
\[
\frac{1}{\|Z_{n_k}\|}(X_{n_k}+Z_{n_k}) \to Z.
\]
Since $X_{n_k}+Z_{n_k}\in\cA$ for all $k\in\N$, we see that $Z\in\cA^\infty$. Moreover, $Z$ satisfies
\[
\pi(Z) = \lim_{k\to\infty}\frac{\pi(Z_{n_k})}{\|Z_{n_k}\|} = \lim_{k\to\infty}\frac{\risk(X_{n_k})}{\|Z_{n_k}\|} = 0
\]
by continuity. However, in contrast to our assumption, this implies that $Z\in\cA^\infty\cap\ker(\pi)$. It follows that $\Risk(\cK)$ must be bounded, concluding the proof.
\end{proof}

\medskip

The absence of scalable good deals is a necessary condition for upper semicontinuity whenever the underlying acceptance set is star-shaped. In this case, upper semicontinuity is equivalent to $\Risk(X)$ being bounded for some position $X\in\cX$ (such that $\Risk(X)$ is nonempty). Recall that, by Assumption 2, $\cA$ is automatically star-shaped when convex.

\begin{theorem}
\label{theo: characterization usc star shaped}
Assume $\cA$ is star-shaped. Then, the following statements are equivalent:
\begin{enumerate}[(a)]
  \item $\Risk$ is upper semicontinuous.
  \item $\Risk(X)$ is bounded for every $X\in\cX$.
  \item $\Risk(X)$ is bounded for some $X\in\cX$ such that $\Risk(X)\neq\emptyset$.
  \item The market admits no scalable good deals, i.e.~$\cA^\infty\cap\ker(\pi)=\{0\}$.
\end{enumerate}
\end{theorem}
\begin{proof}
It follows immediately from Proposition~\ref{prop: characterization usc} that {\em (a)} implies {\em (b)}, which clearly implies {\em (c)}. Moreover, Corollary~\ref{cor: sufficient usc asymptotic} tells us that {\em (d)} implies {\em (a)}. Hence, it suffices to show that $\cA^\infty\cap\ker(\pi)=\{0\}$ whenever $\Risk$ is bounded valued at some position that admits optimal payoffs. To this effect, take any $X\in\cX$ such that $\Risk(X)$ is nonempty and bounded. In this case, we must have $\cA^\infty\cap\ker(\pi)=\Risk(X)^\infty=\{0\}$ due to Proposition~\ref{prop: reformulation Risk} and this concludes the proof of the equivalence.
\end{proof}


\subsection*{Lower semicontinuity}

In this section we focus on the stability notion that, as discussed above, is most relevant in our framework. We start by highlighting that lower semicontinuity can be equivalently stated in terms of sequences. In addition, we show that lower semicontinuity is automatically ensured for all positions once it holds at every point belonging to the intersection between the boundary of $\cA$ and the boundary of the ``augmented'' acceptance set $\cA+\ker(\pi)$, which coincides with the set where $\risk$ is zero.

\begin{proposition}
\label{prop: characterization lower semicontinuity}
The following statements are equivalent:
\begin{enumerate}[(a)]
  \item $\Risk$ is lower semicontinuous at every $X\in\cX$.
  \item $\Risk$ is lower semicontinuous at every $X\in\bd\cA\cap\bd(\cA+\ker(\pi))$.
  \item For every $X\in\cX$ we have
\[
X_n\to X, \ Z\in\Risk(X) \ \implies \ \exists Z_n\in\Risk(X_n) \,:\, Z_n\to Z.
\]
  \item For every $X\in\cX$ we have
\[
X_n\to X, \ Z\in\Risk(X) \ \implies \ \exists Z_{n_k}\in\Risk(X_{n_k}) \,:\, Z_{n_k}\to Z.
\]
\end{enumerate}
\end{proposition}
\begin{proof}
It is clear that {\em (a)} implies {\em (b)} and that {\em (c)} implies {\em (d)}. Moreover, \cite[Theorem~17.21]{AliprantisBorder2006} asserts that {\em (a)} and {\em (d)} are always equivalent. Hence, it remains to prove that {\em (b)} implies {\em (c)}. To this end, assume that {\em (b)} holds and consider arbitrary $(X_n)\subset\cX$ and $X\in\cX$ satisfying $X_n\to X$. In addition, take any $Z\in\Risk(X)$.

\smallskip

Assume first that $X\in\bd\cA\cap\bd(\cA+\ker(\pi))$. In this case, for any $k\in\N$ set $\cU_k=\Int\cB_{\frac{1}{k}}(Z)$ and note that $\Risk(X)\cap\cU_k\neq\emptyset$. Since $\Risk$ is lower semicontinuous at $X$ by assumption, it follows that for any $k\in\N$ there exists a neighborhood $\cV_k\subset\cX$ of $X$ such that $\Risk(Y)\cap\cU_k\neq\emptyset$ for all $Y\in\cV_k$. Observe that, by taking subsequent intersections, we may assume without loss of generality that $\cV_{k+1}\subset\cV_k$ for every $k\in\N$. In addition, for each $k\in\N$ there exists $n(k)\in\N$ such that $X_n\in\cV_k$ for all $n\geq n(k)$. In line with our previous assumption, we may always ensure that $n(k+1)>n(k)$ for every $k\in\N$. As a result, for every $n\in\N$ we can select a payoff $Z_n\in\Risk(X_n)$ in such a way that $Z_n\in\cU_k$ whenever $n\in[n(k),n(k+1))$ for some $k\in\N$ (if $n<n(1)$ we take an arbitrary $Z_n\in\Risk(X_n)$). It is not difficult to verify that $Z_n\to Z$, so that {\em (c)} holds.

\smallskip

If otherwise $X\notin\bd\cA\cap\bd(\cA+\ker(\pi))$, it suffices to define $Y_n=X_n+Z$ for all $n\in\N$ and $Y=X+Z$ and observe that $X+Z\in\bd\cA\cap\bd(\cA+\ker(\pi))$ and
\[
0 = Z-Z \in \Risk(X)-Z = \Risk(X+Z)
\]
by virtue of Proposition~\ref{prop: reformulation Risk}. At this point, one can apply the preceding argument to find $W_n\in\Risk(X_n+Z)$ for all $n\in\N$ such that $W_n\to0$. But then $W_n+Z\in\Risk(X_n)$ for every $n\in\N$ and $W_n+Z\to Z$, showing that {\em (b)} holds also in this case.
\end{proof}

\smallskip

\begin{remark}[{\bf On Assumption 3}]
\label{rem: assumption continuity}
The main objective of this paper is the study of lower semicontinuity of the optimal payoff map. It follows from the preceding proposition that a necessary condition for our question to be meaningful is that $\risk$ is finite and continuous. Indeed, it is not difficult to see that $\Risk$ cannot be lower semicontinuous at any position $X\in\cX$ for which $\Risk(X)\neq\emptyset$ unless $\risk$ is finite and continuous there. This explains why we have to work under Assumption 3.\hfill$\qed$
\end{remark}

\medskip

\begin{remark}[{\bf On sufficient conditions for lower semicontinuity}]
\label{rem: conditions for lsc}
As already mentioned in the introduction, general sufficient conditions for lower semicontinuity are typically hard to establish; see Bank et al.~(1983) and the references therein. Some well-known conditions are the following (a definition of strict lower semicontinuity is given before Lemma~\ref{lem: auxiliary lower semicontinuity}):
\begin{enumerate}[(i)]
  \item There exist a strictly lower semicontinuous map $\cS_1:\cX\rightrightarrows\cM$ and a lower semicontinuous map $\cS_2:\cX\rightrightarrows\cM$ such that $\Risk(X)=\cS_1(X)\cap\cS_2(X)$ for all $X\in\cX$; see Lemma~2.2.5 in Bank et al.~(1983).
  \item The lower section $\{X\in\cX \,; \ Z\in\Risk(X)\}$ is open for all $Z\in\cM$; see Lemma~17.12 in Aliprantis and Border~(2006).
  \item The graph $\{(X,Z)\in\cX\times\cM \,; \ Z\in\Risk(X)\}$ is convex; see Theorem~5.9 in Rockafellar and Wets~(2009).
\end{enumerate}
Unfortunately, as is not difficult to verify, the above conditions are generally not satisfied in our framework. The natural choices for $\cS_1$ and $\cS_2$ in {\em (i)} are given by
\[
\cS_1(X)=\{Z\in\cM \,; \ X+Z\in\cA\} \ \ \ \mbox{and} \ \ \ \cS_2(X)=\{Z\in\cM \,; \ \pi(Z)=\risk(X)\}.
\]
However, neither $\cS_1$ nor $\cS_2$ can be strictly lower semicontinuous. Condition {\em (ii)} is also clearly never satisfied. For condition {\em (iii)} to be fulfilled, $\risk$ must be linear along any segment, i.e.~for all $X,Y\in\cX$ and $\lambda\in[0,1]$ we must have $\risk(\lambda X+(1-\lambda)Y)=\lambda\risk(X)+(1-\lambda)\risk(Y)$. In the common case where $\risk(0)=0$, this would force $\risk$ to be linear on the entire space $\cX$, a condition that is seldom satisfied if $\cM$ is a strict subspace of $\cX$.\hfill$\qed$
\end{remark}


\subsubsection*{Lower semicontinuity for polyhedral acceptance sets}

In this section we prove that lower semicontinuity of the optimal payoff map always holds if the underlying acceptance set is polyhedral.
We start by establishing two useful lemmas. The first simple result highlights a useful property of points of the boundary of $\cA$ (we provide a proof since we were unable to find an explicit reference). Here, we denote by $\relint\cC$ the {\em relative interior} of a set $\cC\subset\cX$, i.e.~the set of all $X\in\cC$ that admit a neighborhood $\cU_X\subset\cX$ satisfying $\cU_X\cap\aff\cC\subset\cC$.

\begin{lemma}
\label{lem: active constraints}
Assume $\cA$ is polyhedral and is represented by $\varphi_1,\dots,\varphi_m\in\cX'_+$. For any $X\in\cX$ define
\[
I_\cA(X) = \{i\in\{1,\dots,m\} \,; \ \varphi_i(X)=\sigma_\cA(\varphi_i)\}\,.
\]
Then, for every nonempty convex subset $\cC\subset\bd\cA$ the following statements hold:
\begin{enumerate}[(i)]
  \item If $X\in\relint\cC$ and $Y\in\cC$, then $I_\cA(X)\subset I_\cA(Y)$.
  \item If $X,Y\in\relint\cC$, then $I_\cA(X)=I_\cA(Y)$.
\end{enumerate}
\end{lemma}
\begin{proof}
Take any $X\in\relint\cC$ and $Y\in\cC$. It is clear that $X\pm\varepsilon(Y-X)\in\cC$ for $\varepsilon>0$ small enough. Then, for any $i\in I_\cA(X)$, we have
\[
\sigma_\cA(\varphi_i)\pm\varepsilon\varphi_i(Y-X) = \varphi_i(X\pm\varepsilon(Y-X)) \geq \sigma_\cA(\varphi_i)\,,
\]
which is only possible if $\varphi_i(Y)=\varphi_i(X)$. This yields $i\in I_\cA(Y)$ and shows that {\em (i)} holds. Assertion {\em (ii)} is a direct consequence of {\em (i)}.
\end{proof}

\medskip

The second preliminary result provides a decomposition of the optimal payoff map in the polyhedral case, which is the key ingredient when establishing lower semicontinuity. Recall that a point $X$ of a convex set $\cC\subset\cX$ is said to be an {\em extreme point of $\cC$} whenever $\cC\setminus\{X\}$ is still convex. Recall also that any polyhedral set admits at most finitely many extreme points; see Lemma~7.78 in Aliprantis and Border~(2006).

\begin{lemma}
\label{lem: decomposition Risk polyhedral}
Assume $\cA$ is polyhedral. Then, there exists $\cC:\cX\rightrightarrows\cM$ such that $\cC(X)\neq\emptyset$ and
\[
\Risk(X) = (\cA^\infty\cap\ker(\pi))+\cC(X)
\]
for every $X\in\cX$ and such that $\cC(\cK)$ is bounded for every compact set $\cK\subset\cX$.
\end{lemma}
\begin{proof}
Let $X\in\cX$ be fixed and recall from Proposition~\ref{prop: reformulation Risk} that, since $\cA$ is polyhedral, $\Risk(X)$ is also polyhedral (in $\cM$). We denote by $\cC(X)$ the convex hull of the set of the extreme points of the polyhedral set $\Risk(X)\cap\cN$, where $\cN$ is any vector subspace of $\cM$ satisfying $\lin\Risk(X)\cap\cN=\{0\}$. Note that $\cN$ does not depend on the choice of $X$ by virtue of Proposition~\ref{prop: reformulation Risk}. Moreover, note that, by construction, $\Risk(X)\cap\cN$ does not contain any vector subspace and, thus, does admit extreme points. Now, it follows from Lemmas~16.2 and~16.3 in Barvinok~(2002) that $\Risk(X)$ can be decomposed as
\[
\Risk(X) = \Risk(X)^\infty+\cC(X)\,.
\]
In view of Proposition~\ref{prop: reformulation Risk}, we can equivalently write
\[
\Risk(X) = (\cA^\infty\cap\ker(\pi))+\cC(X)\,.
\]

\smallskip

It remains to prove that $\cC$ maps compact sets into bounded sets. To this effect, assume that $\cA$ is represented by $\varphi_1,\dots,\varphi_m\in\cX'_+$ and define for every $i\in\{1,\dots,m+2\}$ the maps $\alpha_i:\cN\to\R$ and $\beta_i:\cX\to\R$ by setting
\[
\alpha_i(Z)=
\begin{cases}
\varphi_i(Z) & \mbox{if} \ i\in\{1,\dots,m\},\\
\pi(Z) & \mbox{if} \ i=m+1,\\
-\pi(Z) & \mbox{if} \ i=m+2,
\end{cases}
 \ \ \ \mbox{and} \ \ \
\beta_i(X)=
\begin{cases}
\sigma_\cA(\varphi_i)-\varphi_i(X) & \mbox{if} \ i\in\{1,\dots,m\},\\
\risk(X) & \mbox{if} \ i=m+1,\\
-\risk(X) & \mbox{if} \ i=m+2.
\end{cases}
\]
Then, it follows from~\eqref{eq: polyhedral structure Risk 1} and~\eqref{eq: polyhedral structure Risk 2} that the polyhedral set $\Risk(X)\cap\cN$ can be expressed as
\[
\Risk(X)\cap\cN = \bigcap_{i=1}^{m+2}\{Z\in\cN \,; \ \alpha_i(Z)\geq\beta_i(X)\}
\]
for every $X\in\cX$. Now, denote by $\cI$ the collection of all subsets $I\subset\{1,\dots,m+2\}$ consisting of $d=\dim\cN$ elements and such that $\{\alpha_i \,; \ i\in I\}$ is linearly independent. The collection $\cI$ is nonempty by virtue of Proposition~3.3.3 in Bertsekas et al.~(2003). Moreover, define $\alpha_I:\cN\to\R^d$ and $\beta_I:\cX\to\R^d$ by setting
\[
\alpha_I(Z)=(\alpha_{i_1}(Z),\dots,\alpha_{i_d}(Z)) \ \ \ \mbox{and} \ \ \ \beta_I(X)=(\beta_{i_1}(X),\dots,\beta_{i_d}(X))
\]
for every $I=\{i_1,\dots,i_d\}\in\cI$ and note that $\alpha_I$ is linear and bijective and $\beta_I$ is continuous (due to the continuity of $\risk$) for every $I\in\cI$. As a result of Proposition~3.3.3 in Bertsekas et al.~(2003), every extreme point of $\Risk(X)\cap\cN$, with $X\in\cX$, has the form $\alpha_I^{-1}(\beta_I(X))$ for some $I\in\cI$. This implies that, for any compact set $\cK\subset\cX$, we have
\[
\cC(\cK) \subset \co\left(\bigcup_{I\in\cI}\alpha_I^{-1}(\beta_I(\cK))\right)\,.
\]
Note that $\alpha_I^{-1}(\beta_I(\cK))$ is compact for every choice of $I\in\cI$. Since $\cI$ contains finitely many members and the convex hull of a compact set is still compact, we conclude that $\cC(\cK)$ is contained in a compact set. This establishes that $\cC(\cK)$ is bounded and concludes the proof.
\end{proof}

\medskip

We are now ready to prove that lower semicontinuity always holds if the underlying acceptance set is polyhedral.

\begin{theorem}
\label{theo: lower semicontinuity polyhedral}
Assume $\cA$ is polyhedral. Then, $\Risk$ is lower semicontinuous.
\end{theorem}
\begin{proof}
We assume throughout that $\cA$ is represented by $\varphi_1,\dots,\varphi_m\in\cX'_+$. Take $X\in\cX$ and consider a sequence $(X_n)\subset\cX$ converging to $X$ and $Z\in\Risk(X)$. Let $(Z_n)\subset\cM$ be any sequence satisfying $Z_n\in\cC(X_n)$ for every $n\in\N$, where $\cC:\cX\rightrightarrows\cM$ is the set-valued map from Lemma~\ref{lem: decomposition Risk polyhedral}. Since $(X_n)$ is contained in a compact set, Lemma~\ref{lem: decomposition Risk polyhedral} tells us that $(Z_n)$ is bounded. Hence, passing to a suitable subsequence which we still denote by $(Z_n)$, we have $Z_n\to W$ for some payoff $W\in\cM$. Note that $X_n+Z_n\to X+W$ and $X_n+Z_n\in\bd\cA\cap\bd(\cA+\ker(\pi))$ for all $n\in\N$ by Proposition~\ref{prop: reformulation Risk}. Hence, we infer that $X+W\in\bd\cA\cap\bd(\cA+\ker(\pi))$ or, equivalently, $W\in\Risk(X)$ again by Proposition~\ref{prop: reformulation Risk}.

\smallskip

Recall from Corollary~\ref{cor: existence for polyhedral} that $\Risk(X)\neq\emptyset$. If $|\Risk(X)|=1$, then we must have $W=Z$ and we immediately conclude that $\Risk$ is lower semicontinuous at $X$ by virtue of Proposition~\ref{prop: characterization lower semicontinuity}. Hence, let us assume that $|\Risk(X)|>1$. In this case, being convex, $\Risk(X)$ has nonempty relative interior.

\smallskip

We first assume that $Z\in\relint\Risk(X)$. Since $X+Z\in\relint(X+\Risk(X))$ and $X+W\in X+\Risk(X)$ and since $X+\Risk(X)\subset\bd\cA$ by Proposition~\ref{prop: reformulation Risk}, we infer from Lemma~\ref{lem: active constraints} that
\[
I_\cA(X+Z) \subset I_\cA(X+W)\,.
\]
This, in particular, implies that
\[
I_\cA(X+Z) \subset \{i\in\{1,\dots,m\} \,; \ \varphi_i(Z-W)=0\}\,.
\]
For $i\in I_\cA(X+Z)$ we can use the above inclusion to obtain (note that $X_n+Z_n\in\cA$ for all $n\in\N$)
\[
\varphi_i(X_n+Z_n+Z-W) = \varphi_i(X_n+Z_n) \geq \sigma_\cA(\varphi_i)
\]
for every $n\in\N$. For $i\not\in I_\cA(X+Z)$ we immediately see that
\[
\varphi_i(X_n+Z_n+Z-W) = \varphi_i(X_n+Z_n-X-W)+\varphi_i(X+Z) > \sigma_\cA(\varphi_i)
\]
for $n$ large enough since $\varphi_i(X_n+Z_n-X-W)\to 0$. By combining the above inequalities we infer that $X_n+Z_n+Z-W \in \cA$ for $n$ large enough. Now, set $W_n=Z_n+Z-W$ for every $n\in\N$ and note that $W_n\to Z$. Moreover, we eventually have $X_n+W_n\in\cA$ and
\[
\pi(W_n) = \pi(Z_n)+\pi(Z)-\pi(W) = \risk(X_n)+\risk(X)-\risk(X) = \risk(X_n)\,.
\]
In other words, the sequence $(W_n)\subset\cM$ satisfies $W_n\in\Risk(X_n)$ for $n$ large enough and $W_n\to Z$. This shows that $\Risk$ is lower semicontinuous at $X$ by Proposition~\ref{prop: characterization lower semicontinuity}.

\smallskip

Assume now that $Z\notin\relint\Risk(X)$. In this case, we may approximate $Z$ by elements in the relative interior of $\Risk(X)$ and apply the above argument to each of them. It then follows from Proposition~\ref{prop: characterization lower semicontinuity} that $\Risk$ is lower semicontinuous at $X$ also in this case.
\end{proof}

\medskip

Recall that acceptance sets based on Expected Shortfall or Test Scenarios are polyhedral in a finite-dimensional setting. Hence, the following corollary is a direct consequence of our general result on polyhedral acceptance sets.

\begin{corollary}
\label{cor: lsc es and test scenarios finite dimension}
Assume $\dim(\cX)<\infty$ and $\cA$ is based on either Expected Shortfall or Test Scenarios. Then, $\Risk$ is lower semicontinuous
\end{corollary}


\subsubsection*{Lower semicontinuity for strictly-convex acceptance sets}

The optimal payoff map is always lower semicontinuous if the chosen acceptance set is strictly convex. In fact, any position admits a unique optimal payoff in that case and the map associating to each position its optimal payoff is continuous.

\begin{theorem}
Assume $\cA$ is strictly convex. Then, $\Risk$ is lower semicontinuous.
\end{theorem}
\begin{proof}
Recall that, by Corollary~\ref{cor: uniqueness strictly convex}, every position admits at most one optimal payoff by strict convexity. Hence, $\Risk$ is obviously bounded valued and Theorem~\ref{theo: characterization usc star shaped} implies that $\Risk$ is upper semicontinuous. As a result, we infer from Remark~\ref{rem: semicontinuity} that $\Risk$ is also lower semicontinuous.
\end{proof}


\subsubsection*{Counterexamples to lower semicontinuity}

\begin{example}[{\bf $\VaR$-based acceptance sets}]
We show that, if acceptability is based on $\VaR$, a slight perturbation of a capital position may drastically reduce the range of optimal payoffs. In fact, the number of choices of optimal payoffs could abruptly shrink from infinite to just one.

\smallskip

Fix $\alpha\in\big(0,\tfrac{1}{3}\big)$ (recall that values of $\alpha$ close to $0$ are the interesting ones in practice) and consider a probability space $(\Omega,\cF,\probp)$ and a partition $\{E,F,G\}\subset\cF$ of $\Omega$ satisfying $\probp(E)=\probp(F)=\alpha$ and $\probp(G)=1-2\alpha$. Let $\cX=L^\infty(\Omega,\cF,\probp)$ and consider the $\VaR$-based acceptance set
\[
\cA = \{X\in\cX \,; \ \probp(X<0)\leq\alpha\}.
\]
Moreover, take $\cM=\Span(\one_\Omega,Z)$ for $Z=\one_E-\one_G$ and define $\pi$ by setting $\pi(\one_\Omega)=1$ and $\pi(Z)=0$. Under these specifications, Assumptions 1 to 3 are all satisfied.

\smallskip

Since $\cA\cap\ker(\pi)=\{0\}$ (the market admits no good deals), it follows from Corollary~\ref{cor: existence solutions star shaped} that $\Risk(X)\neq\emptyset$ for all $X\in\cX$. It is not difficult to verify that $\risk(0)=0$ and
\[
\Risk(0) = \{\lambda Z \,; \ \lambda\in\R, \ \probp(\lambda Z<0)\leq\alpha\} = \{\lambda Z \,; \ \lambda\in(-\infty,0]\}.
\]
Now, for each $n\in\N$ consider the position $X_n=-\tfrac{1}{n}\one_F\in\cX$ and note that $\risk(X_n)=0$, so that
\[
\Risk(X_n) = \{\lambda Z \,; \ \lambda\in\R, \ \probp(X_n+\lambda Z<0)\leq\alpha\} = \{0\}.
\]
Since we obviously have $X_n\to0$, we infer that $\Risk$ fails to be lower semicontinuous at $0$. In particular, every position $X_n$ admits a unique optimal payoff whereas the limit position $0$ allows for an infinity of optimal payoffs.\hfill$\qed$
\end{example}

\medskip

The failure of lower semicontinuity we have just illustrated critically depends on the nonconvexity of $\VaR$-based acceptance sets. One may thus wonder whether the above extreme instability behaviour is also possible if the chosen acceptance set is convex. Our next examples show that convexity is not sufficient to ensure lower semicontinuity. In fact, the same extreme instability is compatible with a variety of important (non-polyhedral) convex acceptance sets.

\begin{example}[{\bf Scenario-based acceptance sets (in infinite dimension)}]
\label{ex: lsc test scenarios}
In a finite-dimensional setting, acceptance sets based on Test Scenarios are polyhedral and the corresponding optimal payoff map is thus lower semicontinuous by virtue of Theorem~\ref{theo: lower semicontinuity polyhedral}. This applies, in particular, to the positive cone. The picture changes dramatically once we move to an infinite-dimensional setting, where polyhedrality ceases to hold. In this case, we show that a slight perturbation of the underlying financial position may cause the set of optimal payoffs to shrink from a rich infinite set to a mere singleton.

\medskip

Consider a non-atomic probability space $(\Omega,\cF,\probp)$ and fix an event $E\in\cF$ with $\probp(E)>0$. Let $(E_k)\subset\cF$ be a countable partition of $E$ satisfying $\probp(E_k)>0$ for all $k\in\N$, which always exists by non-atomicity. We work in the space $\cX=L^\infty(\Omega,\cF,\probp)$, which is partially ordered by the canonical almost-sure ordering, and consider the scenario-based acceptance set
\[
\cA = \{X\in\cX \,; \ X\one_E\geq0\}.
\]
Moreover, we consider the space of eligible payoffs $\cM=\Span(\one_\Omega,Z)$, where
\[
Z = -\one_{E_2}+\sum_{k\geq3}\tfrac{1}{k}\one_{E_k}.
\]
The pricing functional is defined by $\pi(\one_\Omega)=1$ and $\pi(Z)=0$. Under these specifications, Assumptions 1 to 3 are all satisfied.

\smallskip

Since $\cA\cap\ker(\pi)=\{0\}$ (the market admits no good deals), Corollary~\ref{cor: existence solutions star shaped} implies that $\Risk(X)\neq\emptyset$ for all $X\in\cX$. Now, fix $\gamma\geq0$ and define a position $X\in\cX$ by setting
\[
X = \gamma\one_{E_2}+\sum_{k\geq3}\tfrac{1}{k}\one_{E_k}.
\]
A direct computation shows that $\risk(X)=0$. Moreover, it is not difficult to verify that
\begin{equation}
\label{eq: counterexample worst case 1}
\Risk(X) = \{\lambda Z \,; \ \lambda\in\R, \ (X+\lambda Z)\one_E\geq0\} = \{\lambda Z \,; \ \lambda\in[-1,\gamma]\}.
\end{equation}
For any $n\in\N$ consider the position $X_n\in\cX$ given by
\[
X_n = \gamma\one_{E_2}+\sum_{k=3}^{2+n}\tfrac{1}{k}\one_{E_k}+
\sum_{k\geq3+n}\tfrac{1}{k^2}\one_{E_k}.
\]
Note that for each $n\in\N$ we have
\[
\|X_n-X\|_\infty = \sup_{k\geq3+n}\|(X_n-X)\one_{E_k}\|_\infty = \sup_{k\geq3+n}\tfrac{1}{k}-\tfrac{1}{k^2} < \tfrac{1}{3+n},
\]
which implies $X_n\to X$. A direct computation shows that $\risk(X_n)=0$ and
\begin{equation}
\label{eq: counterexample worst case 2}
\Risk(X_n) = \{\lambda Z \,; \ \lambda\in\R, \ (X_n+\lambda Z)\one_E\geq0\} = \{\lambda Z \,; \ \lambda\in[0,\gamma]\}
\end{equation}
for all $n\in\N$. In particular, for any $n\in\N$ and $\lambda\in\R$ the inequality $(X_n+\lambda Z)\one_E\geq0$ yields
\[
\lambda \geq \sup_{k\geq3+n}\big\{-\tfrac{1}{k^2}\big(\tfrac{1}{k}\big)^{-1}\big\} = -\inf_{k\geq3+n}\tfrac{1}{k} = 0.
\]
As a result, we see that $\Risk(X_n)=\Risk(X_1)$ for all $n\in\N$. However, $\Risk(X)$ is strictly larger than $\Risk(X_1)$. In particular, $\Risk(X)$ is infinite whereas $\Risk(X_1)$ consists of a single payoff if we choose $\gamma=0$. This clearly shows that $\Risk$ cannot be lower semicontinuous at $X$.\hfill$\qed$
\end{example}

\medskip

The next example shows that the failure of lower semicontinuity is the rule, rather than the exception, under convex acceptance sets.

\begin{example}[{\bf Convex law-invariant acceptance sets (in infinite dimension)}]
We work in the setting of Example~\ref{ex: lsc test scenarios} and set $E=\Omega$. Note that we can always choose the partition $(E_k)\subset\cF$ so as to have $\probp(E_1)>\alpha$. We consider the same space of eligible payoffs and a convex acceptance set $\cA\subset\cX$ satisfying Assumption 2 and such that
\begin{equation}
\label{eq: counterexample law invariant}
\cA \subset \{X\in\cX \,; \ \ES_\alpha(X)\leq0\}
\end{equation}
for some $\alpha\in(0,1)$. This means that $\cA$ is more stringent than some $\ES$-based acceptance set. Note that, since the space of eligible payoffs has not changed, Assumption 1 is clearly satisfied and Assumption 3 follows from Proposition~\ref{prop: finiteness continuity rho}.

\smallskip

Of course, condition~\eqref{eq: counterexample law invariant} holds for any acceptance sets based on ES. More generally, it follows by combining Proposition~1.1 in Svindland~(2010) and Theorem~4.67 in F\"{o}llmer and Schied~(2011) that condition~\eqref{eq: counterexample law invariant} is fulfilled by any acceptance set $\cA\subset\cX$ that is convex and law-invariant ($X\in\cA$ for all $X\in\cX$ having the same probability distribution of some element of $\cA$) and that satisfies
\[
\cA \subset \{X\in\cX \,; \ \VaR_\alpha(X)\leq0\}.
\]

\smallskip

Since $\cA\cap\ker(\pi)=\{0\}$ (the market admits no good deals), Corollary~\ref{cor: existence solutions star shaped} implies that $\Risk(X)\neq\emptyset$ for all $X\in\cX$. As a preliminary observation, take $\lambda\in\R$ and note that for any $Y\in\cX$ with $Y\one_{E_1}=Z\one_{E_1}$ we have
\[
\ES_\alpha(Y+\lambda Z)
\begin{cases}
=0 & \mbox{if} \ Y+\lambda Z\geq0,\\
>0 & \mbox{otherwise}.
\end{cases}
\]
This is because $\probp(Y+\lambda Z=0)\geq\probp(E_1)>\alpha$. Since $\cX_+\subset\cA\subset\{X\in\cX \,; \ \ES_\alpha(X)\leq0\}$, we infer that
\begin{equation}
\label{eq: counterexample ES}
Y+\lambda Z\in\cA \ \iff \ Y+\lambda Z\geq0.
\end{equation}
Then, it follows from~\eqref{eq: counterexample worst case 1} that $\risk(X)=0$ and
\[
\Risk(X) = \{\lambda Z \,; \ \lambda\in\R, \ X+\lambda Z\in\cA\} = \{\lambda Z \,; \ \lambda\in[-1,\gamma]\}.
\]
Similarly, for any $n\in\N$ we infer from~\eqref{eq: counterexample worst case 2} that $\risk(X_n)=0$ and
\[
\Risk(X_n) = \{\lambda Z \,; \ \lambda\in\R, \ X_n+\lambda Z\in\cA\} = \{\lambda Z \,; \ \lambda\in[0,\gamma]\}.
\]
We can therefore argue as in Example~\ref{ex: lsc test scenarios} to conclude that $\Risk$ cannot be lower semicontinuous at $X$. In particular, choosing $\gamma=0$, we see that a small perturbation of $X$ may cause the set of optimal payoffs to abruptly shrink from an infinite set to just a singleton.\hfill$\qed$
\end{example}

\smallskip

\begin{remark}
With the exception of $\VaR$-based acceptance sets, all the above examples cannot be adapted to a finite-dimensional setting. In particular, Corollary~\ref{cor: lsc es and test scenarios finite dimension} tells us that we always have lower semicontinuity for acceptance sets based on $\ES$ and Test Scenarios in finite dimension. One may therefore wonder whether convexity is capable of ensuring lower semicontinuity at least in a finite-dimensional setting. This is, however, far from being true as shown in Baes and Munari~(2017).\hfill$\qed$
\end{remark}


\subsubsection*{Lower semicontinuity of nearly-optimal payoff maps}

The preceding examples show that, as soon as we depart from polyhedrality, the optimal payoff map may fail to be lower semicontinuous, in which case the choice of the optimal portfolio is affected by severe instability. This is irrespective of the underlying acceptance set being convex or not. It is therefore natural to turn to the study of lower semicontinuity for {\em nearly-optimal payoff maps}. For any given $\varepsilon>0$ these are set-valued maps $\Riskeps:\cX\rightrightarrows\cM$ defined by setting
\[
\Riskeps(X) := \{Z\in\cM \,; \ X+Z\in\cA, \ \pi(Z)<\risk(X)+\e\}.
\]
Instead of focusing on optimal payoffs, we relax the optimality condition and look for all the payoffs that ensure acceptability at a cost that is close to being minimal. The parameter $\e$ defines the range of tolerance around the optimal cost. In the language of parametric optimization, $\Riskeps$ is referred to as the {\em $\e$-optimal set mapping}.

\medskip

The study of nearly-optimal set mappings is a recurrent theme in parametric optimization and the key result on lower semicontinuity is Theorem~4.2.4 in Bank et al.~(1983). After adapting that result to our framework, we exploit it to establish a variety of lower semicontinuity results for nearly-optimal payoff maps. To this effect, we first need the following preliminary lemma, which provides a simple generalization of Lemma~2.2.5 and Corollary~2.2.5.1 in Bank et al.~(1983). Here, for $X\in\cX$ we say that a set-valued map $\cS:\cX\rightrightarrows\cM$ is {\em strictly lower semicontinuous at $X$} if for any $Z\in\cS(X)$ there exist open neighborhoods $\cU_X\subset\cX$ of $X$ and $\cU_Z\subset\cM$ of $Z$ such that
\[
Y\in\cU_X \ \implies \ \cU_Z\subset\cS(Y)\,.
\]
We say that $\cS$ is {\em strictly lower semicontinuous} if the above property holds for every $X\in\cX$. Clearly, strict lower semicontinuity is (typically much) stronger than lower semicontinuity.

\begin{lemma}
\label{lem: auxiliary lower semicontinuity}
For any maps $\cS_1,\cS_2:\cX\rightrightarrows\cM$ the following statements hold:
\begin{enumerate}[(i)]
  \item Assume $\cS_1$ is strictly lower semicontinuous and $\cS_2$ is lower semicontinuous. Then, the set-valued map $\cS:\cX\rightrightarrows\cM$ given by $\cS(X)=\cS_1(X)\cap\cS_2(X)$ is lower semicontinuous.
  \item Assume $\cS_1$ is strictly lower semicontinuous and $\cS_1(X)\subset\cS_2(X)\subset\Cl\cS_1(X)$ for all $X\in\cX$. Then, $\cS_2$ is lower semicontinuous.
\end{enumerate}
\end{lemma}
\begin{proof}
{\em (i)} Fix $X\in\cX$ and assume $\cS(X)\cap\cU\neq\emptyset$ for some open set $\cU\subset\cM$. Take $Z\in\cS(X)\cap\cU$ and note that, by strict lower semicontinuity, we find open neighborhoods $\cU_X\subset\cX$ of $X$ and $\cU_Z\subset\cM$ of $Z$ such that $\cU_Z\subset\cS_1(Y)$ for all $Y\in\cU_X$. Since $Z\in\cS_2(X)\cap\cU\cap\cU_Z$, there exists a neighborhood $\cV_X\subset\cX$ of $X$ such that $Y\in\cV_X$ implies $\cS_2(Y)\cap\cU\cap\cU_Z\neq\emptyset$ by lower semicontinuity. As a result, it follows that
\[
\cS(Y)\cap\cU = \cS_1(Y)\cap\cS_2(Y)\cap\cU \supset \cU_Z\cap\cS_2(Y)\cap\cU \neq \emptyset
\]
for every $Y\in\cU_X\cap\cV_X$, proving that $\cS$ is lower semicontinuous.

\smallskip

{\em (ii)} Fix $X\in\cX$. Assume that $\cS_2(X)\cap\cU\neq\emptyset$ for some open set $\cU\subset\cM$ and take $Z\in\cS_2(X)\cap\cU$. Then, $Z\in\Cl\cS_1(X)\cap\cU$ and we can therefore find a sequence $(Z_n)\subset\cS_1(X)$ such that $Z_n\to Z$. In particular, there exists $k\in\N$ for which $Z_k\in\cU$. By strict lower semicontinuity, we find open neighborhoods $\cU_X\subset\cX$ of $X$ and $\cU_k\subset\cM$ of $Z_k$ such that $\cU_k\subset\cS_1(Y)$ for all $Y\in\cU_X$. Then, it follows that
\[
Z_k \in \cU_k\cap\cU \subset \cS_1(Y)\cap\cU \subset \cS_2(Y)\cap\cU
\]
for any $Y\in\cU_X$, showing that $\cS_2$ is lower semicontinuous at $X$.
\end{proof}

\smallskip

\begin{proposition}
\label{prop: lsc subptimal payoffs}
Fix $X\in\cX$ and assume the set-valued map $\cF:\cX\rightrightarrows\cM$ defined by
\[
\cF(X) := \{Z\in\cM \,; \ X+Z\in\cA\}
\]
is lower semicontinuous at $X$. Then, $\Riskeps$ is lower semicontinuous at $X$ for any $\e>0$.
\end{proposition}
\begin{proof}
For any fixed $\e>0$ consider the set-valued map $\cH:\cX\rightrightarrows\cM$ given by
\[
\cH(X) = \{Z\in\cM \,; \ \pi(Z)<\risk(X)+\e\}.
\]
It is straightforward to see that $\cH$ is strictly lower semicontinuous at every point by the continuity of $\pi$ and $\risk$. Since $\Riskeps(X)=\cF(X)\cap\cH(X)$ for all $X\in\cX$, the assertion is a direct consequence of Lemma~\ref{lem: auxiliary lower semicontinuity}.
\end{proof}

\smallskip

\begin{remark}
In parametric optimization the map $\cF$ is usually referred to as the {\em constraint set mapping}. In our setting, it can be viewed as a generalized version of the {\em set-valued risk measures} introduced in Jouini et al.~(2004) and further studied in Hamel and Heyde (2010).\hfill$\qed$
\end{remark}

\medskip

The next theorem is our main result on the stability of nearly-optimal payoff maps.

\begin{theorem}
\label{theo: epsilon lower semicontinuity}
Assume $\Cl(\Int\cA)=\cA$ and the set-valued map $\cG:\cX\rightrightarrows\cM$ defined by
\[
\cG(X) := \{Z\in\cM \,; \ X+Z\in\Int\cA\}
\]
satisfies $\cG(X)\neq\emptyset$ for all $X\in\cX$. Then, $\Riskeps$ is lower semicontinuous for every $\e>0$.
\end{theorem}
\begin{proof}
For any arbitrary $X\in\cX$ we have $\cG(X)\subset\cF(X)\subset\Cl(\cG(X))$. The first inclusion is obvious. To establish the second inclusion, note that
\[
\cF(X) = \cM\cap(\cA-X) = \cM\cap(\Cl(\Int\cA)-X) \subset \Cl(\cM\cap(\Int\cA-X)) = \Cl(\cG(X)).
\]
Moreover, note that $\cG$ is strictly lower semicontinuous. To show this, take $X\in\cX$ and $Z\in\cG(X)$. Since $X+Z\in\Int\cA$, we find open neighborhoods $\cU_X\subset\cX$ of $X$ and $\cU_Z\subset\cM$ of $Z$ such that $\cU_X+\cU_Z\subset\Int\cA$. This implies that $W\in\cG(Y)$ for any $Y\in\cU_X$ and $W\in\cU_Z$, showing that $\cG$ is indeed strictly lower semicontinuous at $X$. Hence, we are in a position to apply Lemma~\ref{lem: auxiliary lower semicontinuity} to infer that $\cF$ is lower semicontinuous. The assertion is now a direct consequence of Proposition~\ref{prop: lsc subptimal payoffs}.
\end{proof}

\medskip

The above density condition is always satisfied in model spaces where the positive cone has nonempty interior. This implies, in particular, that nearly-optimal payoff maps are always lower semicontinuous in finite-dimensional spaces or in spaces of bounded random variables whenever some eligible payoff lies in the interior of the positive cone.

\begin{corollary}
Assume $\Int\cX_+\cap\cM\neq\emptyset$. Then, $\Riskeps$ is lower semicontinuous for every $\e>0$.
\end{corollary}
\begin{proof}
Note first that $\cA$ has nonempty interior since it contains $\cX_+$. Take any $X\in\cA$ and $Z\in\Int\cX_+\cap\cM$ and set $X_n=X+\tfrac{1}{n}Z$ for all $n\in\N$. Clearly, we have $X_n\to X$. We claim that $X_n\in\Int\cA$ for any fixed $n\in\N$. To see this, consider the neighborhood of $Z$ defined by
\[
\cU_Z = \{Y\in\cX \,; \ 2Z\geq Y\geq 0\}.
\]
Then, we easily see that $X+\tfrac{1}{n}\cU_Z$ is a neighborhood of $X_n$ contained in $\cX_+$ and, hence, in $\cA$. This establishes the claim and proves that $\Cl(\Int\cA)=\cA$. In view of Theorem~\ref{theo: epsilon lower semicontinuity}, to conclude the proof it suffices to show that $\cG(X)\neq\emptyset$ for all $X\in\cX$. To this effect, take any $X\in\cX$ and note that $Z+\lambda X\in\Int\cX_+$ for $\lambda>0$ small enough. This yields $X+\tfrac{1}{\lambda}Z\in\Int\cX_+$ and shows that $\cG(X)$ is not empty.
\end{proof}

\medskip

The last part of this section deals with convex acceptance sets. In this case, the density condition $\Cl(\Int\cA)=\cA$ is well-known to be satisfied (provided $\cA$ has nonempty interior). For nearly-optimal payoff maps to be lower semicontinuous in the convex case, it is therefore sufficient that every position can be made ``strictly acceptable'', i.e.~can be moved into the interior of the acceptance set, by means of a suitable eligible payoff. The next results describe some situations where this can be achieved.

\begin{corollary}
\label{cor: epsilon lower semicontinuity asymptotic cone}
Assume $\cA$ is convex and $\Int(\cA^\infty)\cap\cM\neq\emptyset$. Then, $\Riskeps$ is lower semicontinuous for every $\e>0$.
\end{corollary}
\begin{proof}
Fix $\varepsilon>0$ and let $Z\in\Int(\cA^\infty)\cap\cM$. Then, for any $X\in\cX$ we find $\lambda>0$ small enough so that $\lambda X+Z\in\Int(\cA^\infty)$. Since $\cA^\infty$ is a cone, this is equivalent to $X+\tfrac{1}{\lambda}Z\in\Int(\cA^\infty)$ and shows that $\cG(X)\neq\emptyset$. As a result of Theorem~\ref{theo: epsilon lower semicontinuity}, we conclude that $\Risk$ is $\varepsilon$-lower semicontinuous at $X$.
\end{proof}

\medskip

In the next result we denote by $\cX_{++}$ the set of strictly-positive elements in $\cX$. Recall that $X\in\cX_+$ is strictly positive if $\varphi(X)>0$ for all nonzero functional $\varphi\in\cX'_+$.

\begin{corollary}
Assume $\cA$ is convex, $\Int(\cA^\infty)\neq\emptyset$ and $\cX_{++}\cap\cM\neq\emptyset$. Then, $\Riskeps$ is lower semicontinuous for every $\e>0$.
\end{corollary}
\begin{proof}
Let $Z\in\cX_{++}\cap\cM$. The assertion is an immediate consequence of Corollary~\ref{cor: epsilon lower semicontinuity asymptotic cone} once we show that $Z\in\Int(\cA^\infty)$. To this effect, assume that $Z\notin\Int(\cA^\infty)$. In this case, by Hahn-Banach, we find a nonzero functional $\varphi\in\cX'$ satisfying (note that $0\in\cA^\infty$)
\begin{equation}
\label{eq: epsilon lower semicontinuity strictly positive}
\varphi(Z) \leq \sigma_{\cA^\infty}(\varphi) \leq 0\,.
\end{equation}
Since $\cA^\infty$ is a cone, Lemma~\ref{prop: properties support functions} implies that $\varphi(X)\geq0$ for all $X\in\cA^\infty$. In particular, since $\cX_+\subset\cA$, we see that $\varphi\in\cX'_+$. This yields $\varphi(Z)>0$ by strict positivity, which, however, contradicts~\eqref{eq: epsilon lower semicontinuity strictly positive}. As a result, we conclude that $Z$ must belong to $\Int(\cA^\infty)$.
\end{proof}

\end{document}